\documentclass[11pt,leqno,twoside]{amsart}
\usepackage{amssymb,amsmath,amsthm,soul,color}
\usepackage{t1enc}
\usepackage[cp1250]{inputenc}
\usepackage{a4,indentfirst,latexsym}
\usepackage{graphics}
\usepackage{mathrsfs}
\usepackage{cite,enumitem,graphicx}
\usepackage[colorlinks=true,urlcolor=blue,
citecolor=red,linkcolor=blue,linktocpage,pdfpagelabels,
bookmarksnumbered,bookmarksopen]{hyperref}
\usepackage[english]{babel}
\usepackage[left=2.61cm,right=2.61cm,top=2.72cm,bottom=2.72cm]{geometry}
\usepackage[metapost]{mfpic}
\usepackage[normalem]{ulem}

\newtheorem{Th}{Theorem}[section]
\newtheorem{Prop}[Th]{Proposition}
\newtheorem{Lem}[Th]{Lemma}
\newtheorem{Cor}[Th]{Corollary}

\newenvironment{altproof}[1]
{\noindent
	{\em Proof of {#1}}.}
{\nopagebreak\mbox{}\hfill $\Box$\par\addvspace{0.5cm}}

\newcommand{\vp}{\varphi}

\newcommand{\eps}{\varepsilon}

\def\div{\mathop{\mathrm{div}\,}}

\def\supp{\mathrm{supp}}

\def\id{\mathrm{id}}

\def\Z{\mathbb{Z}}

\def\N{\mathbb{N}}
\def\R{\mathbb{R}}

\def\cl{\mathrm{cl\,}}

\def\U{\mathcal{U}} 

\def\D{\mathcal{D}}

\newcommand{\cC}{{\mathcal C}}
\newcommand{\cD}{{\mathcal D}}

\newcommand{\cK}{{\mathcal K}}

\newcommand{\cM}{{\mathcal M}}
\newcommand{\cN}{{\mathcal N}}
\newcommand{\cO}{{\mathcal O}}
\newcommand{\cP}{{\mathcal P}}

\newcommand{\cS}{{\mathcal S}}

\newcommand{\cU}{{\mathcal U}}

\newcommand{\Om}{\Omega}

\newcommand{\weakto}{\rightharpoonup}

\def\id{\mathrm{id}}

\newcommand{\tu}{\widetilde{u}}
\newcommand{\tv}{\widetilde{v}}

\numberwithin{equation}{section}

\begin{document}
	\title{General class of optimal Sobolev inequalities and nonlinear scalar field equations}
\author[J. Mederski]{Jaros\l aw Mederski}
\address[J. Mederski]{\newline\indent
	Institute of Mathematics,
	\newline\indent 
	Polish Academy of Sciences,
	\newline\indent 
	ul. \'Sniadeckich 8, 00-656 Warsaw, Poland
	\newline\indent 
	and
	\newline\indent 
	Department of Mathematics,
	\newline\indent 
	Karlsruhe Institute of Technology (KIT), 
	\newline\indent 
	D-76128 Karlsruhe, Germany
}
\email{\href{mailto:jmederski@impan.pl}{jmederski@impan.pl}}
	\maketitle
	
	\pagestyle{myheadings} \markboth{\underline{J. Mederski}}{
		\underline{Nonlinear scalar field equations}}

\begin{abstract}
We find a class of optimal Sobolev inequalities 
$$\Big(\int_{\mathbb{R}^N}|\nabla u|^2\, dx\Big)^{\frac{N}{N-2}}\geq C_{N,G}\int_{\mathbb{R}^N}G(u)\, dx, \quad u\in\mathcal{D}^{1,2}(\mathbb{R}^N), N\geq 3,$$
where the nonlinear function $G:\mathbb{R}\to\mathbb{R}$ of class $\mathcal{C}^1$ satisfies general growth assumptions in the spirit of the fundamental works of Berestycki and Lions. We admit, however, a wider class of problems involving zero, positive and infinite mass cases as well as $G$ need not be even.
We show that any minimizer is radial up to a translation, moreover, up to a dilation, it is a least energy solution of the nonlinear scalar field equation
$$-\Delta u = g(u)\quad \hbox{in }\mathbb{R}^N,\quad\hbox{with }g=G'.$$ 
In particular, if $G(u)=u^2\log |u|$, then the sharp constant is $C_{N,G}:=2^*(\frac{N}{2})^{2^*}e^{\frac{2(N-1)}{N-2}}(\pi)^{\frac{N}{N-2}}$ and  $u_\lambda(x)=e^{\frac{N-1}{2}-\frac{\lambda^2}{2}|x|^2}$ with $\lambda>0$ constitutes the whole family of mini\-mizers up to translations. The optimal inequality provides a new proof of the classical logarithmic Sobolev inequality based on a Pohozaev manifold approach. Moreover, if $N\geq 4$, then there is at least one nonradial solution and if, in addition, $N\neq 5$, then there are infinitely many nonradial solutions of the nonlinear scalar field equation. The energy functional associated with the problem may be infinite on $\mathcal{D}^{1,2}(\mathbb{R}^N)$ and is not Fr\'echet diffe\-rentiable in its domain. We present a variational approach to this problem based on a new variant of Lions' lemma in $\mathcal{D}^{1,2}(\mathbb{R}^N)$.
\end{abstract}

{\bf MSC 2010:} Primary: 35J20, 58E05

{\bf Key words:} Nonlinear scalar field equations, logarithmic Sobolev inequality, cubic-quintic effect, critical point theory, nonradial solutions, concentration compactness, Lions' lemma, Pohozaev manifold, zero mass case, infinite mas case.

\section*{Introduction}
\setcounter{section}{1}

In view of the classical Sobolev inequality one can show that there is a constant $C_{N,G}>0$ such that the following inequality
\begin{equation}\label{eq:ineq}
\Big(\int_{\R^N}|\nabla u|^2\, dx\Big)^{\frac{N}{N-2}}\geq C_{N,G}\int_{\R^N}G(u)\, dx
\end{equation}
holds for all $u\in\cD^{1,2}(\R^N)$, where $\cD^{1,2}(\R^N)$ stands for the completion of $\cC_0^{\infty}(\R^N)$ with respect to the norm
$\|u\|=\Big(\int_{\R^N}|\nabla u|^2\,dx\Big)^{\frac12}$, $N\geq 3$, and $G$ satisfies
the following assumptions 
 \begin{itemize}
 	\item[(g0)] $g:\R \to \R$ is continuous, $g(0)=0$,
 	 $G(s)=\int_0^s g(t)\, dt$,
 	 $G_+(s)=\int_0^s \max\{g(t),0\}\, dt$  for $s\geq 0$ and $G_+(s)=\int_{s}^0 \max\{-g(t),0\}\, dt$ for $s<0$.
 	\item[(g1)] $\lim_{s\to 0}G_+(s)/|s|^{2^*}=0$, where $2^*=\frac{2N}{N-2}$.
 	\item[(g2)] There exists  $\xi_0>0$ such that $G(\xi_0)>0$.
	\item[(g3)] $\lim_{|s|\to \infty}G_+(s)/|s|^{2^*}=0$ and 
 	 	$\limsup_{|s|\to \infty}|g(s)|/|s|^{2^*-1}<\infty$.
 \end{itemize}

We show that \eqref{eq:ineq} is optimal, that is the equality holds in \eqref{eq:ineq} for some $u\neq 0$, and then $u$ is called  a {\em minimizer}. Observe that, if $u$ is a minimizer, then $u(\lambda\cdot)$ and $u(\cdot +y)$ are minimizers for any $\lambda>0$ and $y\in\R^N$. The first main result reads as follows.
\begin{Th} \label{ThMain1} Suppose that (g0)--(g3) are satisfied.\\
	(a) There is a radially symmetric solution $u\in\cD^{1,2}(\R^N)$ of
	\begin{equation}
	\label{eq}
	-\Delta u= g(u)\quad \hbox{in }\R^N
	\end{equation}
	such that $u\in \cM$ and
	$J(u)=\inf_{\cM} J>0$,
	where $J$ is the associated energy functional
	\begin{equation}\label{eq:action}
	J(u)=\frac12\int_{\R^N} |\nabla u|^2- \int_{\R^N} G(u)\, dx,
	\end{equation}
	and
	\begin{equation}\label{def:Poh}
	\cM=\Big\{u\in \cD^{1,2}(\R^N)\setminus\{0\}: \int_{\R^N}|\nabla u|^2=2^*\int_{\R^N}G(u)\, dx\Big\}.
	\end{equation}
	If in addition $g$ is odd, then $u$ is positive.\\
	(b) If $u\in\cM$ and $J(u)=\inf_{\cM} J$, then $u$ is a radial (up to a translation) solution to \eqref{eq}.\\
	(c) The optimal constant in \eqref{eq:ineq} is 
	$$C_{N,G}=2^*\Big(\frac12-\frac{1}{2^*}\Big)^{-\frac{2}{N-2}}(\inf_{\cM} J)^{\frac{2}{N-2}}.$$
	Moreover, if $u\in\cM$ and $J(u)=\inf_{\cM} J$, then
	$u$ is a minimizer of \eqref{eq:ineq}. If $u$ is a minimizer of \eqref{eq:ineq}, then $u (\lambda\cdot)\in\cM$ and $J(u(\lambda\cdot))=\inf_{\cM} J$ for a unique $\lambda>0$.
	In particular, there is a  radially symmetric solution of \eqref{eq} such that the equality holds in \eqref{eq:ineq}.
\end{Th}
Using standard arguments we show that any (weak) solution $u$ of \eqref{eq} such that $G(u)\in L^1(\R^N)$ satisfies the Pohozaev identity
\begin{equation}\label{eq:Poho}
\int_{\R^N}|\nabla u|^2=2^*\int_{\R^N}G(u)\, dx,
\end{equation}
see Proposition \ref{PropPohozaev}.
Hence the {\em Pohozaev  manifold} $\cM$ contains all nontrivial finite energy solutions, and $u$ obtained in Theorem \ref{ThMain1} (a) is a {\em least energy solution}. Moreover if, in addition, 
$$G(s)\leq  -c_1s^2+c_2s^{2^*}$$ 
for some constants $c_1,c_2>0$, for instance in the positive mass case below \eqref{eq:pmc}, then \eqref{eq:Poho} implies that $u\in \cM\subset H^1(\R^N)$.\\
\indent If $g$ is odd, then positive and radially symmetric solutions of \eqref{eq} have been considered by Berestycki and Lions in their fundamental papers \cite{BerLions,BerLionsII} and multiplicity of radial solutions have been given in \cite{BerLionsII,BerLionsInfZero}. In fact, if we do not look for least energy solutions, then by the strong maximum principle we can solve \eqref{eq} under the following more general growth assumption introduced in \cite{BerLionsInfZero}:
\begin{itemize}
	\item[(g3')] Let $\xi_1:=\inf\{\xi>0: G(\xi)>0\}$. If $g(s)>0$ for all $s>\xi_1$, then
	 	\begin{equation*}
	 	\lim_{s\to \infty}G(s)/s^{2^*}=0,\hbox{ and }\limsup_{s\to \infty}g(s)/s^{2^*-1}<\infty.
	 	\end{equation*}
\end{itemize}
Namely, suppose that $g$ is odd and satisfies (g0)--(g2) and (g3').
Similarly as in \cite{BerLions}, we modify $g$ in the following way. If $g(s)>0$ for all $s>\xi_0$, then $\tilde{g}=g$. Otherwise we set
$\xi_1:=\inf\{\xi\geq \xi_0: g(\xi)\leq 0\}$, 
\begin{equation*}
\tilde{g}(s)=\left\{
\begin{array}{ll}
g(s)&\quad\hbox{if }0\leq s \leq \xi_1,\\
g(\xi_1) &\quad\hbox{if }s>\xi_1,
\end{array}
\right.
\end{equation*}
and $\tilde{g}(s)=-\tilde{g}(-s)$ for $s<0$. Hence $\tilde{g}$ satisfies assumptions (g0)--(g3) of Theorem 1.1 and by the strong maximum principle if $u\in \cD^{1,2}(\R^N)$ solves $-\Delta u=\tilde{g}(u)$, then $|u(x)|\leq \xi_1$ and $u$ is a solution of \eqref{eq}. However, it is not clear whether $J(u)=\inf_{\cM}J$ and $u$ is a least energy solution under assumptions (g0)--(g2) and (g3'). So far, a positive, radially symmetric and {\em least energy} solution has been obtained in \cite{BerLions}[Theorem 3] in the {\em positive mass case} for the modified nonlinearity $\tilde{g}$. Namely, instead of (g1), we have
\begin{equation}\label{eq:pmc}
-\infty<\liminf_{s\to 0}g(s)/s\leq\limsup_{s\to 0}g(s)/s=-m<0,
\end{equation}
and after the above modification of $g$, in fact, it has been assumed that 
\begin{eqnarray}\label{g3strong}
\lim_{|s|\to \infty}g(s)/|s|^{2^*-1}=0,
\end{eqnarray}
also in other works \cite{JeanjeanTanaka,MederskiBL,JeanjeanLu}. The latter condition excludes some important examples in physics, which are taken into account under our assumptions (g0)--(g3). For instance, take
\begin{equation}\label{eq:CubicQuintic}
g(s)=|s|^{p-2}s-|s|^{2^*-2}s-ms,\quad 2<p<2^*,
\end{equation}
and note that $g$ satisfies (g0)--(g3) if and only if $m\in (0,m_0)$, where
$$m_0:= \frac{(N-2)(2^*-p)}{N(p-2)}\Big(\frac{N(p-2)}{2p}\Big)^{\frac{2^*-2}{2^*-p}}.$$
In this work we are able to find a least energy solution minimizing the energy on $\cM$ under more general assumptions, in particular  \eqref{g3strong} could be violated, we no longer require that $g$ is odd as well as we may consider an {\em infinite mass case} with $m=\infty$.  Firstly we present the following simple corollary.
\begin{Cor}\label{th:cubicquintic}
Suppose that $g$ is given by \eqref{eq:CubicQuintic}.\\
(a) For any $m\in (0,m_0)$ there is a unique positive and radially symmetric solution $u$ of \eqref{eq}  minimizing $J$ on $\cM\subset H^1(\R^N)$, which is also a minimizer of \eqref{eq:ineq}.\\
(b) If $m\notin (0,m_0)$, then \eqref{eq} has only trivial finite energy solution.
\end{Cor}
In a particular case $N=3$ and $p=4$, Corollary \ref{th:cubicquintic} was obtained by  Killip et al. in \cite{Killip}[Theorem 2.2.(i)], and we arrive at the cubic-quintic problem  which appears e.g. in nonlinear optics or in the the study of Bose–Einstein condensates \cite{Gammal,Mihalache}. In general case $N\geq 3$, the uniqueness and the non-degeneracy of positive solutions have been proved recently by  Lewin and Rota Nodari in \cite{Lewin}. In this paper we show, in addition,  that $u$ is a minimizer of $J$ on $\cM$ and, if $N\geq 4$, we there are nonradial sign-changing solutions to \eqref{eq:CubicQuintic}, see Corollary \ref{th:cubicquinticlog2} below.\\
\indent The relation between solutions to \eqref{eq} and minimizers of an inequality of the form \eqref{eq:ineq} was already presented e.g. in \cite{BerLions,BerLionsII} or more recently by Byeon, Jeanjean and Mariş in \cite[Lemma 1]{ByeonJeanjeanMaris} provided that any solution to \eqref{eq} satisfies the Pohozaev identity. In our situation, only finite energy solutions satisfies the Pohozaev identity (Proposition \ref{PropPohozaev}) and our assumptions admit also the {\em infinite mass case}, i.e. $m=\infty$ in \eqref{eq:pmc}. Together with the existence result in Theorem \ref{ThMain1}, we would like to show that such relation between solutions to \eqref{eq} and minimizers of \eqref{eq:ineq}  allows to provide also a new proof of the classical {\em logarithmic Sobolev inequality} given in \cite{Weissler}:
\begin{equation}\label{eq:ineqLogSob}
\frac{N}{4}\log\Big(\frac{2}{\pi e N}\int_{\R^N}|\nabla u|^2\,dx\Big)\geq  \int_{\R^N}|u|^2\log(|u|)\,dx,\quad\hbox{for }u\in H^1(\R^N), \int_{\R^N}|u|^2\,dx=1,
\end{equation}
which is equivalent to the Gross inequality \cite{Gross}. Indeed,
note that the following  nonlinearity
\begin{equation}\label{eq:logSob}
G(s)=s^2\log |s|\quad\hbox{for }s\neq 0,\hbox{ and }G(0)=0
\end{equation}
is in the {\em infinite mass case} and satisfies (g0)--(g3). Therefore,
in view of Theorem \ref{ThMain1} there is a positive and radially symmetric solution of \eqref{eq} with 
\begin{equation}\label{eq:lognon}
g(s)=2s\log |s| + s.
\end{equation}
The Gausson \cite{Bialynicki}
 \begin{equation}\label{eq:Gausson}
 u_1(x)=e^{\frac{N-1}{2}-\frac{1}{2}|x|^2}
 \end{equation}
solves \eqref{eq} and in view of Serrin and Tang \cite{SerrinTang} (cf. \cite{dAvenia,DelPino,DelPinoJMPA}), $u_1$  is a unique positive and  radial solution of \eqref{eq} up to a translation. Thus, one easily verifies that 
$J(u_1)=\big(\frac{1}{2}-\frac{1}{2^*}\big)e^{N-1}\frac{N}{2}(\pi)^{\frac{N}{2}}=
\frac{1}{2}e^{N-1}(\pi)^{\frac{N}{2}}=\inf_{\cM}J$ and by Theorem \ref{ThMain1} (c)
\begin{equation}\label{def:CNG}
C_{N,G}:=2^*\Big(\frac{N}{2}\Big)^{\frac{2}{N-2}}e^{\frac{2(N-1)}{N-2}}(\pi)^{\frac{N}{N-2}}.
\end{equation}
Moreover $u_1$ is a unique minimizer of \eqref{eq:ineq} solving \eqref{eq} up to a translation. Now observe that
 \eqref{eq:ineq}  is equivalent to
 \begin{equation}\label{eq:scaled}
\Big(\int_{\R^N}|\nabla u|^2\,dx\Big)^{\frac{N}{N-2}}\geq C_{N,G}\max_{\alpha\in\R}\Big\{ e^{-\alpha2^*}\int_{\R^N}G(e^{\alpha}u)\,dx\Big\},\quad\hbox{ for } u\in\cD^{1,2}(\R^N),
 \end{equation}
and the equality holds if and only if $u=e^{\beta}u_1(\lambda\cdot)$ for some $\beta\in\R$, $\lambda>0$ and up to a translation.
Assuming that $\int_{\R^N}u^2\, dx=1$,
the maximum of the right hand side of \eqref{eq:scaled} is attained at $\alpha=\frac{N-2}{4}-\int_{\R^N}|u|^2\log |u|\,dx$. Hence, taking into account \eqref{def:CNG} we verify that \eqref{eq:scaled} is equivalent to \eqref{eq:ineqLogSob} provided that $\int_{\R^N}|u|^2\,dx=1$.
Moreover, \eqref{eq:ineqLogSob} is sharp and the family $\lambda^{\frac{N}{2}}u_1(\lambda\cdot)$, $\lambda>0$ consists of unique minimizers up to translations.\\ 
\indent Recall that the optimality of \eqref{eq:ineqLogSob} and the characterization of minimizers have been already proved by Carlen \cite{Carlen} in the context of the Gross inequality as well as by del Pino and Dolbeault \cite{DelPino,DelPinoJMPA} for the interpolated Gagliardo–Nirenberg inequalities and the $L^p$-Sobolev logarithmic inequality. A generalization of the optimal Gross inequality in the context of Orlicz spaces is given by Adams \cite{Adams}. The optimal inequality \eqref{eq:ineq} can be also regarded as a generalization of \eqref{eq:ineqLogSob} and note that we do not need any structural assumptions in the Orlicz setting as in \cite{Adams}. We would like to also mention that Wang and Zhang \cite{WangZhang} have recently provided another proof of the logarithmic Sobolev inequality due to Lieb and Loss \cite{LiebLoss} based on an approximation by minimizers of the classical Sobolev inequalities.\\
\indent In order to solve \eqref{eq} under assumptions (g0)--(g3), we consider the associated energy functional $J:\cD^{1,2}(\R^N)\to\R\cup\{\infty\}$ given by \eqref{eq:action} and observe that $J$ may be infinite on a dense subset of $\cD^{1,2}(\R^N)$. We look for weak solutions of \eqref{eq}, i.e. $J'(u)(v)=0$ for any $v\in\cC_0^{\infty}(\R^N)$, however, $J$ cannot be Fr\'echet differentiable and this is the first main difficulty in comparison to the the positive mass case \eqref{eq:pmc} studied e.g. in \cite{BerLions,BerLionsII,MederskiBL,JeanjeanLu,JeanjeanTanaka}. Note that in the positive mass case and under assumption \eqref{g3strong}, $J$ is well-defined, of class $\cC^1$ on $H^1(\R^N)$ and Jeanjean and Tanaka \cite{JeanjeanTanaka} showed that the least energy solution obtained in \cite{BerLions} minimizes the energy on the {\em Pohozaev manifold} $\cM$  defined by \eqref{def:Poh} in $H^1(\R^N)$. In Theorem \ref{ThMain1} (a) we prove that there is a least energy solution minimizing $J$ on the Pohozaev manifold $\cM$  under more general assumptions (g0)--(g3) including also the zero mass case ($m=0$) as well as the infinite mass case ($m=\infty$), e.g. \eqref{eq:logSob}. We also present a simple approach of finding minimizers on $\cM$, which is equivalent to finding minimizers of \eqref{eq:ineq} -- see Section \ref{sec:proof}, in particular  Lemma \ref{lem:theta}.\\
\indent Note that in \cite{MederskiBL} the positive mass case has been studied, and if $N\geq 4$ nonradial solutions have been found there. Next, Jeanjean and Lu \cite{JeanjeanLu} have recently provided a mountain pass approach and reproved the main results from \cite{MederskiBL} based on the monotonicity trick \cite{Jeanjean}. Therefore, our next aim is to show that the similar results hold under assumptions (g0)--(g3) and we give an answer to the  problem \cite{BerLionsII}[Section 10.8] concerning the existence and multiplicity of nonradial solutions of \eqref{eq} also in the zero mass case as well as in the infinite mass case.\\
\indent Namely, let $N\geq 4$ and similarly as in \cite{BartschWillem}, let us fix $\tau\in\cO(N)$ such that $\tau(x_1,x_2,x_3)=(x_2,x_1,x_3)$ for $x_1,x_2\in\R^m$ and $x_3\in\R^{N-2m}$, where $x=(x_1,x_2,x_3)\in\R^N=\R^m\times\R^m\times \R^{N-2m}$ and $2\leq m\leq N/2$.
We define
\begin{equation}\label{eq:DefOfX}
X_{\tau}:=\big\{u\in \cD^{1,2}(\R^N): u(x)=-u(\tau x)\;\hbox{ for all }x\in\R^N\big\}.
\end{equation}
Clearly, if $u\in X_\tau$ is radial, i.e. $u(x)=u(\rho x)$ for any $\rho \in\cO(N)$, then $u=0$. Hence $X_\tau$ does not contain nontrivial radial functions. Then $\cO_1:=\cO(m)\times \cO(m)\times\id \subset \cO(N)$ acts isometrically on $\cD^{1,2}(\R^N)$ and let $\cD^{1,2}_{\cO_1}(\R^N)$ denote the subspace of invariant functions with respect to $\cO_1$.
\begin{Th}\label{ThMain2}
If $g$ is odd and $N\geq 4$, then there is a solution  $u\in \cM\cap X_\tau\cap \cD^{1,2}_{\cO_1}(\R^N)$ of \eqref{eq} such that 
	\begin{equation}\label{eq:thmain2}
	J(u)=\inf_{\cM\cap X_\tau\cap \cD^{1,2}_{\cO_1}(\R^N)}J>2\inf_{\cM} J.
	\end{equation}
\end{Th}
Clearly, we infer that problem \eqref{eq} with \eqref{eq:CubicQuintic} or with \eqref{eq:logSob} has a nonradial solution for $N\geq 4$. Moreover,
if, in addition, $N\neq 5$, then we find infinitely many nonradial solutions. Indeed,  we may assume that $N-2m\neq 1$ and let us consider $\cO_2:=\cO(m)\times \cO(m)\times\cO(N-2m)\subset \cO(N)$ acting isometrically on $\cD^{1,2}(\R^N)$ with the subspace of invariant function denoted by $\cD^{1,2}_{\cO_2}(\R^N)$.
\begin{Th}\label{ThMain3}
If $g$ is odd, $N\geq 4$ and $N\neq 5$,  then the following statements hold.\\
(a) There is a solution  $u\in \cM\cap X_\tau\cap \cD^{1,2}_{\cO_2}(\R^N)$ of \eqref{eq} such that 
\begin{equation}\label{eq:thmain3}
J(u)=\inf_{\cM\cap X_\tau\cap \cD^{1,2}_{\cO_2}(\R^N)}J\geq \inf_{\cM\cap X_\tau\cap H^1_{\cO_1}(\R^N)}J.
\end{equation}
(b) 
There is an infinite sequence of solutions $(u_n)\subset \cM\cap X_\tau\cap \cD^{1,2}_{\cO_2}(\R^N)$ to \eqref{eq} such that $J(u_n)\to\infty$ as $n\to\infty$.
\end{Th}

As a consequence of Theorems \ref{ThMain2} and  \ref{ThMain3} we obtain:
\begin{Cor}\label{th:cubicquinticlog2}
	Suppose that $N\geq 4$, and $g$ is given by \eqref{eq:CubicQuintic} with $m\in (0,m_0)$ or $g$ is the logarithmic nonlinearity \eqref{eq:lognon}.\\
	(a)   Then there is a non-radial and sign-changing solution to \eqref{eq} in $\cM\cap X_\tau\cap \cD^{1,2}_{\cO_1}(\R^N)$.\\
	(b) If $N\neq 5$, then there is an infinite sequence of non-radial and sign-changing solutions $(u_n)\subset \cM\cap X_\tau\cap \cD^{1,2}_{\cO_2}(\R^N)$ to \eqref{eq} such that $J(u_n)\to\infty$ as $n\to\infty$.
\end{Cor}

Note that there is little work on the problem \eqref{eq} involving the zero or infinite mass case expressed by general assumptions without Ambrosetti-Rabinowitz-type condition \cite{AR}, or any monotonicity behaviour. The first difficulty is that $J$ may be infinite and is not Fr\'echet differentiable in its domain. The second one is related with the lack of compactness of the problem in $\R^N$; even if we find a Palais-Smale sequence, we do not know whether the sequence is bounded and contains a (weakly) convergent subsequence. Berestycki and Lions in \cite{BerLions}  minimized  $u\mapsto \int_{\R^N}|\nabla u|^2\, dx$ on  the constraint of radial functions such that
$G(u)\in L^1(\R^N)$ and $\int_{\R^N}G(u)\, dx=1$.
In order to get multiplicity of solutions they approximated the zero mass case $g$ by suitable functions $g_\eps$ in the positive mass case, i.e. $-g'_\eps(0)>0$ and $g_\eps\to g$ uniformly on compact subsets of $\R$ as $\eps\to 0^+$. Using results of \cite{BerLionsII} they solved the approximated problem in the positive mass case.  Letting $\eps\to 0$, a sequence of radial solutions of \eqref{eq} have been obtained. Another approach based on approximations of $\cD^{1,2}_{\cO(N)}(\R^N)$ by $\big\{u\in \cD^{1,2}_{\cO(N)}(\R^N): u(x)=0\hbox{ for }|x|\geq L\big\}$ for $L\to\infty$ is due to 
 Struwe \cite{StruweMA}. Observe that in all these works the radial symmetry plays an important role, since one gets the uniform decay at infinity of functions from $\cD^{1,2}_{\cO(N)}(\R^N)$ (see \cite{BerLions}[Radial Lemma A.III]) and the the compactness lemma of Strauss \cite{BerLions}[Lemma A.I] is applicable. In the nonradial setting these arguments are no longer available.\\
\indent Now we sketch our approach with a new and simple approximation $J_\eps$ of $J$. 
Let $g_{+}(s)=G_+'(s)$, $g_{-}(s):=g_{+}(s)-g(s)$ and $G_{-}(s):=G_{+}(s)-G(s)\geq 0$ for $s\in\R$.
In view of (g3), $G_{+}(u)\in L^{1}(\R^N)$ for $u\in\cD^{1,2}(\R^N)\subset L^{2^*}(\R^N)$, however
 $G_{-}(u)$ may not be integrable
unless $G_{-}(u)\leq c|u|^{2^*}$ for some $c>0$. In order to overcome this problem, for any $\eps\in [0,1)$ let us take an  even function $\vp_\eps:\R\to [0,1]$ such that $\vp_\eps(s)=\frac{1}{\eps^{2^*-1}}|s|^{2^*-1}$ for $|s|\leq \eps$, $\vp_{\eps}(s)=1$ for $|s|\geq \eps$. We introduce a new functional
\begin{equation}\label{eq:actionEPS}
J_\eps(u)=\frac12\int_{\R^N} |\nabla u|^2+\int_{\R^N}  G_{-}^\eps(u)\, dx-\int_{\R^N} G_{+}(u)\, dx,
\end{equation}
 such that $G_{-}^\eps(s)=\int_0^s \vp_\eps(t)g_-(t)\, dt$, $s\in\R$,
and now observe that $G_{-}^\eps(s)\leq c_\eps|s|^{2^*}$ for any $s\in\R$ and some constant $c_\eps>0$ depending on $\eps>0$.
 Hence, for $\eps\in (0,1)$,  $J_\eps$ is well-defined on $\cD^{1,2}(\R^N)$, continuous and $J_\eps'(u)(v)$ exists for any $u\in\cD^{1,2}(\R^N)$ and $v\in\cC_0^\infty(\R^N)$. Hence we call $u$ a {\em critical point} of $J_\eps$ provided that $J_\eps'(u)(v)=0$ for any $v\in\cC_0^\infty(\R^N)$. Next,
we show that any minimizing sequence of $J_\eps$ on the following Pohozaev manifold
\begin{equation}\label{def:PohEPS}
\cM_\eps=\Big\{u\in \cD^{1,2}(\R^N)\setminus\{0\}: \int_{\R^N}|\nabla u|^2=2^*\int_{\R^N}G_{+}(u)-G_{-}^\eps(u)\, dx\Big\}
\end{equation}
converges to a nontrivial critical point $u_\eps$ of $J_\eps$ up to a subsequence and up to a translation -- see Lemma \ref{lem:theta}. The last argument requires the following variant of the classical Lions' lemma \cite{Lions84}, \cite{Willem}[Lemma 1.21] applied to $\Psi=G_{+}$ satisfying \eqref{eq:Psi}.
\begin{Lem}\label{lem:Lions}
	Suppose that $(u_n)\subset \cD^{1,2}(\R^N)$ is bounded and for some $r>0$ 	
	\begin{equation}\label{eq:LionsCond11}
	\lim_{n\to\infty}\sup_{y\in \R^N} \int_{B(y,r)} |u_n|^2\,dx=0.
	\end{equation}
	Then  
	$$\int_{\R^N} \Psi(u_n)\, dx\to 0\quad\hbox{as } n\to\infty$$
	for any continuous $\Psi:\R\to [0,\infty)$ such that
	\begin{equation}
	\label{eq:Psi}
	\lim_{s\to 0} \frac{\Psi(s)}{|s|^{2^*}}=\lim_{|s|\to\infty} \frac{\Psi(s)}{|s|^{2^*}}=0.
	\end{equation}
\end{Lem}
Note that concentration-compactness arguments in the zero mass case have been conside\-red so far in more restrictive settings e.g. in \cite{ClappMaia}[Lemma 3.5] or \cite{BenForAzzAprile}[Lemma 2], where one has to require that $\Psi(s)\leq c\min\{|s|^p,|s|^q\}$ for some $2<p<2^*<q$ and constant $c>0$. Condition \eqref{eq:Psi} seems to be optimal and we prove Lemma \ref{lem:Lions} in Section \ref{sec:Lions}, see also Lemma \ref{lem:Conv}. \\
\indent Having found a critical point $u_\eps\in\cM_\eps$ of the approximated functional $J_\eps$, we let $\eps\to 0$ and passing to a subsequence we obtain a solution of \eqref{eq} in Theorem \ref{ThMain1}.  Next, repeating the similar arguments, we prove Theorem \ref{ThMain2} as well as Theorem \ref{ThMain3} (a) in the nonradial setting.  Note that this is a simpler approach in comparison to \cite{MederskiBL,JeanjeanLu} and it seems that we cannot argue directly as in these papers, since we do not require \eqref{eq:pmc} and \eqref{g3strong}, which are crucial for decompositions of Palais-Smale sequences in \cite{JeanjeanLu} and for the variant of Palais-Smale condition \cite{MederskiBL}[$(M)_\beta\,(i)$]. We expect that the approach presented in this paper based on minimization on a Pohozaev manifold with the compactness properties of Lemma \ref{lem:theta} as well as Lions' type results in the spirit of Lemma \ref{lem:Lions} allows to study nonlinear elliptic problems involving also different operators, e.g. \cite{MederskiSiemianowski}.\\
\indent In order to prove the multiplicity result in Theorem \ref{ThMain3} (b), we employ the critical point theory from \cite{MederskiBL}[Section 2]. Namely we observe that there is a homeomorphism $m:\cU\to\cM_\eps$ such that
$$\cU:=\Big\{u\in \cD^{1,2}(\R^N): \int_{\R^N} |\nabla u|^2\, dx=1 \hbox{ and }\int_{\R^N}G_{+}(u)-G_{-}^\eps(u)\, dx>0\Big\}.$$
We show that $J_\eps\circ m :\cU\to \R$ is still of class $\cC^1$. 
The advantage of working with $J_\eps\circ m$ is that $\cU$ is an open subset of a manifold of class $\cC^{1,1}$ and we can use a critical point theory based on the deformation lemma involving a Cauchy problem on $\cU$. This is not feasible on $\cM_\eps$, since $\cM_\eps$ need not be of class $\cC^{1,1}$. We show that $J_\eps\circ m$ satisfies the Palais-Smale condition in $\cU\cap \cD^{1,2}_{\cO_2}(\R^N)$ and we find an unbounded sequence of critical points. This requires a next approximation of $J_\eps$ described in Section \ref{sec:proof2}.
Similarly as above, letting $\eps\to 0$ we prove Theorem \ref{ThMain3} (b). Based on this work, under assumptions (g0)--(g3) one can obtain an unbounded sequence of radial solutions in $\cM\cap \cD^{1,2}_{\cO(N)}(\R^N)$ as well, which was considered already in \cite{BerLionsInfZero,StruweMA}. However this is a different approach, which does not require the radial lemma of Strauss \cite{Strauss,BerLions} -- details are left for the reader.

\section{Concentration-compactness in subspaces of $\cD^{1,2}(\R^N)$}\label{sec:Lions}

We prove the following result, which implies the variant of Lions's lemma in $\cD^{1,2}(\R^N)$.
 \begin{Lem}\label{lem:Conv}
 	Suppose that   $(u_n)\subset \D^{1,2}(\R^N)$ is bounded. Then $u_n(\cdot+y_n)\weakto 0$ in $\D^{1,2}(\R^N)$ for any $(y_n)\subset \Z^N$ if and only if
 	\begin{equation}\label{eq:PsiZero}
 	\int_{\R^N} \Psi(u_n)\, dx\to 0\quad\hbox{as } n\to\infty
 	\end{equation}
 	for any continuous $\Psi:\R\to [0,\infty)$ satisfying \eqref{eq:Psi}.
 \end{Lem}
 \begin{proof}
 	Let $(u_n)\subset \D^{1,2}(\R^N)$ be such that $u_n(\cdot+y_n)\weakto 0$ in $\D^{1,2}(\R^N)$ for any $(y_n)\subset \Z^N$.
 	Take any $\eps>0$ and $2<p<2^*$ and suppose that $\Psi$ satisfies \eqref{eq:Psi}. Then we find $0<\delta<M$ and $c_\eps>0$ such that 
 	\begin{eqnarray*}
 	\Psi(s)&\leq& \eps |s|^{2^*}\quad\hbox{ for }|s|\leq \delta,\\
 	\Psi(s)&\leq& \eps |s|^{2^*}\quad\hbox{ for }|s|>M,\\
 	\Psi(s)&\leq& c_\eps |s|^{p}\quad\hbox{ for }|s|\in (\delta,M].
 	\end{eqnarray*}
 	Let us define $w_n(x):=|u_n(x)|$ for $|u_n(x)|>\delta$ and $w_n(x):=|u_n(x)|^{2^*/2}\delta ^{1-2^*/2}$ for $|u_n(x)|\leq \delta$. Then
 	$(w_n)$ is bounded in $H^1(\R^N)$ and by the Sobolev inequality one has 
 	\begin{eqnarray*}
 		\int_{\Om+y}\Psi(u_n)\,dx&=&
 		\int_{(\Om+y)\cap\{\delta<|u_n|\leq M\}}\Psi(u_n)\,dx
 		+\int_{(\Om+y)\cap (\{|u_n|> M\}\cup \{|u_n|\leq \delta\})}\Psi(u_n)\,dx\\
 		&\leq &
 		c_\eps\int_{(\Om+y)\cap\{\delta<|u_n|\leq M\}}|w_n|^p\,dx
 		+\eps \int_{(\Om+y)\cap (\{|u_n|> M\}\cup \{|u_n|\leq \delta\})}|u_n|^{2^*}\,dx\\
 		&\leq &c_\eps C\Big(\int_{\Om+y}|\nabla w_n|^2+|w_n|^{2}\,dx \Big)\Big(\int_{\Om+y}|w_n|^{p}\,dx\Big)^{1-2/p}+\eps \int_{\Om+y}|u_n|^{2^*}\,dx,
 	\end{eqnarray*}
 	for every $y\in \R^N$, where $\Om=(0,1)^N$ and $C>0$ is a constant.
 	Then we sum the inequalities over $y\in\Z^N$ and we get  
 	$$\begin{aligned}
 	\int_{\R^N}\Psi(u_n)\,dx
 	&\leq c_\eps C\Big(\int_{\R^N}|\nabla w_n|^2+|w_n|^{2}\,dx \Big)\Big(\sup_{y\in\Z^N}\int_{\Om}|w_n(\cdot +y)|^{p}\,dx\Big)^{1-2/p}+\eps \int_{\R^N}|u_n|^{2^*}\,dx.
 	\end{aligned}$$
 	Let us take $(y_n)\subset\Z^N$ such that
 	$$\sup_{y\in\Z^N}\int_{\Om}|w_n(\cdot+y)|^{p}\,dx\leq 2\int_{\Om}|w_n(\cdot+y_n)|^{p}\,dx$$
 	for any $n\geq 1$.
 	Note that
 	$u_n(\cdot+y_n)\weakto 0$ in $\cD^{1,2}(\R^N)$ and passing to a subsequence  we obtain  $u_n(\cdot+y_n)\to 0$ in $L^p(\Om)$. Since $|w_n(x)|\leq |u_n(x)|$,  we infer that $w_n(\cdot+y_n)\to 0$ in $L^p(\Om)$. 
 	Therefore 
 	$$\limsup_{n\to\infty}\int_{\R^N}\Psi(u_n)\,dx\leq \eps \limsup_{n\to\infty}\int_{\R^N}|u_n|^{2^*}\,dx,$$
 	and since $\eps>0$ is arbitrary, we conclude \eqref{eq:PsiZero}.
 On the other hand, suppose that  $u_n(\cdot+y_n)$ does not converges to $0$ for some $(y_n)\subset\Z^N$ and \eqref{eq:PsiZero} holds. We may assume that $u_n(\cdot+y_n)\to u_0\neq 0$ in $L^p(\Om)$ for some bounded domain $\Om\subset\R^N$ and $2<p<2^*.$ Take any $\eps>0$, $q>2^*$ and $\Psi(s):=\min\{|s|^p,\eps^{p-q}|s|^q\}$ for $s\in\R$. Then
 \begin{eqnarray*}
 \int_{\R^N} \Psi(u_n)\, dx&\geq& \int_{\Om\cap \{|u_n|\geq \eps\}}|u_n|^p\, dx+\int_{\Om\cap \{|u_n|\leq \eps\}} \eps^{q-p}|u_n|^q\, dx\\
 &=& \int_{\Om} |u_n|^p\, dx+\int_{\Om\cap \{|u_n|\leq \eps\}} \eps^{p-q}|u_n|^q-|u_n|^p\, dx\\
 &\geq& \int_{\Om} |u_n|^p\, dx-2\eps^p |\Om|.\\
 \end{eqnarray*}
Thus we get $u_n\to 0$ in $L^p(\Om)$ and this contradicts $u_0\neq 0$.
 \end{proof}
 
 \begin{altproof}{Lemma \ref{lem:Lions}}
 	Suppose that there is $(y_n)\subset \Z^N$ such that $u_n(\cdot+y_n)$ does not converge weakly to $0$ in $\D^{1,2}(\R^N)$. Since  $u_n(\cdot+y_n)$ is bounded, then there is $u_0\neq 0$ such that, up to a subsequence,
 	$$u_n(\cdot+y_n)\weakto u_0$$
 	as $n\to\infty$. We find $y\in \R^N$ such that $u_0\chi_{B(y,r)}\neq 0$ in $L^2(B(y,r))$. 
 	Note that passing to a subsequence $u_n(\cdot+y_n)\to u_0$ in $L^2(B(y,r))$. Then, in view of \eqref{eq:LionsCond11}
 	$$\int_{B(y,r)} |u_n(\cdot+y_n)|^2\,dx=\int_{B(y_n+y,r)} |u_n|^2\,dx\to 0$$
 	as $n\to\infty$, which contradicts the fact $u_n(\cdot+y_n)\to u_0\neq 0$ in $L^2(B(y,r))$. Therefore $u_n(\cdot+y_n)\weakto 0$ in $\D^{1,2}(\R^N)$ for any $(y_n)\subset \Z^N$ and by Lemma \ref{lem:Conv} we conclude.
 \end{altproof}

Let us consider $x=(x^1,x^2,x^3)\in\R^N=\R^m\times\R^m\times \R^{N-2m}$ with $2\leq m\leq N/2$ 
such that $x^1,x^2\in\R^m$ and $x^3\in\R^{N-2m}$. Let $\cO_1=\cO(m)\times \cO(m)\times\id\subset \cO(N)$. Then for
$\cO_1$ invariant functions we get the following corollary, whose proof is postponed to Appendix and follows from Proposition \ref{PropLionsCompatib}.

\begin{Cor}\label{CorLions1}
	Suppose that $(u_n)\subset \cD^{1,2}_{\cO_1}(\R^N)$ is bounded, $r_0>0$ is such that for all $r\geq r_0$
	\begin{equation}\label{eq:LionsCond12_cor}
	\lim_{n\to\infty}\sup_{z\in \R^{N-2m}} \int_{B((0,0,z),r)} |u_n|^2\,dx=0.
	\end{equation}
	Then 
	$$\int_{\R^N} \Psi(u_n)\, dx\to 0\quad\hbox{as } n\to\infty$$
	for any continuous function $\Psi:\R\to [0,\infty)$ such that \eqref{eq:Psi} holds.
\end{Cor}

\section{Proofs of Theorem \ref{ThMain1} and Corollary \ref{th:cubicquintic}}\label{sec:proof}

We prove the following Pohozaev type result using a truncation argument due to Kavain, cf. \cite{Struwe}[Lemma 3.5] and \cite{Willem}[Theorem B.3].
\begin{Prop}\label{PropPohozaev} Let $u\in\D^{1,2}(\R^N)$ be a weak solution of \eqref{eq}.
	Then $u\in W^{2,q}_{loc}(\R^N)$ for any $q<+\infty$,  and
	\begin{equation}
	\int_{\R^N}|\nabla u|^2\, dx=2^*\int_{\R^N}G(u)\,dx
	\end{equation}
	provided that $G_{-}(u)\in L^1(\R^N)$.
\end{Prop}
\begin{proof}
Since
	$$|g(u)|\leq c(1+|u|^{2^*-1})$$ 
	for $u\in\R$ and for some constant $c>0$, by Brezis and Kato theorem \cite{BrezisKato} we infer that $u\in W^{2,q}_{loc}(\R^N)$ for any $q<+\infty$.
	Let $\vp\in \cC^\infty_0(\R)$ be such that $0\leq \vp\leq 1$, $\vp(r)=1$ for $r\leq 1$ and $\vp(r)=0$ for $r\geq 2$. Similarly as in \cite{Willem}[Theorem B.3] we define $\vp_n\in \cC^\infty_0(\R^N)$ by the following formula
	$$\vp_n(x)=\vp\Big(\frac{|x|^2}{n^2}\Big).$$ 
	Then there exists $C>0$ such that
	$$\vp_n(x)\leq C,\hbox{ and }|x||\nabla \vp_n(x)|\leq C$$
	for every $n$ and $x\in\R^N$. Recall that 
	\begin{eqnarray*}
		\Delta u \vp_n\langle x, \nabla u\rangle &=& \div\Big(\vp_n(\nabla u\langle x,\nabla u\rangle-x\frac{|\nabla u|^2}{2})\Big)+\frac{N-2}{2}\vp_n|\nabla u|^2\\
		&&-\langle \nabla\vp_n,\nabla u\rangle\langle x,\nabla u\rangle+\langle\nabla\vp_n,x\rangle\frac{|\nabla u|^2}{2}.
	\end{eqnarray*}
	Then by the divergence theorem it is standard to show that
	\begin{eqnarray*}
		\frac{N-2}{2}\int_{\R^N}\vp_n|\nabla u|^2\,dx&=&
		\int_{\R^N}-\langle \nabla\vp_n,\nabla u\rangle\langle x,\nabla u\rangle+\langle\nabla\vp_n,x\rangle\frac{|\nabla u|^2}{2}\, dx\\
		&&+N\int_{\R^N}\vp_n G(u)\, dx 
	   +\int_{\R^N} \langle\nabla\vp_n,x\rangle G(u)\, dx.
	\end{eqnarray*}
	Since $\langle \nabla\vp_n, x\rangle$ is bounded,  $\langle \nabla\vp_n, x\rangle\to 0$ as $n\to\infty$ and $G(u)\in L^1(\R^N)$, then by the Lebesgue dominated convergence theorem  we get 
	$$\int_{\R^N}-\langle \nabla\vp_n,\nabla u\rangle\langle x,\nabla u\rangle+\langle\nabla\vp_n,x\rangle\frac{|\nabla u|^2}{2}\, dx
	+\int_{\R^N} \langle\nabla\vp_n,x\rangle G(u)\, dx\to 0$$
as $n\to\infty$. Since $\vp_n(x)\to 1$
	and we get the required equality. 
\end{proof}

Let $X:=\D^{1,2}(\R^N)$, $\eps>0$ and we set $G_\eps(s):=G_{+}(s)- G_{-}^\eps(s)$, $g_\eps(s):=G_\eps'(s)$ for $s\in\R$. Note that $J_\eps$ is  continuous and $J_\eps'(u)(v)$ exists for any $u\in X$ and $v\in\cC_0^\infty(\R^N)$. 
Moreover let
	\begin{eqnarray*}
		\cM_\eps&:=&\Big\{u\in X\setminus\{0\}: \int_{\R^N}|\nabla u|^2-2^*\int_{\R^N}G_\eps(u)\, dx=0\Big\},\\
		\cS&:=&\Big\{u\in X: \|u\|=1\Big\},\;
		\cP:=\Big\{u\in X: \int_{\R^N}G_\eps(u)\, dx>0\Big\},\\
		\cU&:=&\cS\cap\cP.
	\end{eqnarray*}
\begin{Prop}\label{prop:defMPSU}
The following holds for $\eps>0$.\\
	(i) $\cP$ is open and nonempty. Moreover there is a map $m_\cP:\cP\to \cM_\eps$ such that $m_\cP(u)=u(r\cdot)\in\cM_\eps$ with 
	\begin{equation}\label{eq:defOfR}
	r=r(u)=\frac{\Big(2^*\int_{\R^N}G_\eps(u)\, dx\Big)^{1/2}}{\|u\|}>0.
	\end{equation}
	(ii)  $m:=m_\cP|_{\cU}:\cU\to \cM_\eps$ is a homeomorphism with the inverse $m^{-1}(u)=u(\|u\|^{\frac{2}{N-2}} \cdot)$, $J_\eps\circ m_{\cP}:\cP\to \R$ is continuous and
	\begin{eqnarray*}
	(J_\eps\circ m_\cP)'(u)(v)&=&J'_\eps(m_\cP(u))(v(r(u)\cdot)\\
	&=&r(u)^{2-N}\int_{\R^N}\langle \nabla u,  \nabla v\rangle\,dx-r(u)^{-N}\int_{\R^N}g_\eps(u)v\, dx
	\end{eqnarray*}
	for $u\in\cP$ and $v\in \cC_0^{\infty}(\R^N)$.\\
	(iii) $J_\eps$ is coercive on $\cM_\eps$, i.e. for $(u_n)\subset \cM_\eps$, $J_\eps(u_n)\to\infty$ as $\|u_n\|\to\infty$,  and
	\begin{equation}\label{eq:infJM}
	c_\eps:=\inf_{\cM_\eps} J_\eps=\inf_{\cU} J_\eps\circ m>0.
	\end{equation}
	(iv) If $u_n\to u$, $u_n\in \cU$ and $u\in\partial\cU$, where the boundary of $\cU$ is taken in $\cS$, then $(J_\eps\circ m)(u)\to\infty$ as $n\to\infty$.
\end{Prop}
\begin{proof}
Similarly as in \cite{BerLions}[page 325] or in \cite{MederskiBL}[Remark 4.2] we check that $\cP\neq\emptyset$. Next, we easily verify  (i)--(iv), e.g.  arguing as in the positive mass case in \cite{MederskiBL}[Proposition 4.1].
\end{proof}

The following lemma is crucial and allows to avoid  the analysis of decompositions of Palais-Smale sequences required in \cite{MederskiBL,JeanjeanLu}.

\begin{Lem}\label{lem:theta}
	Suppose that $(u_n)\subset \cM_\eps$, $J_\eps(u_n)\to c_\eps$ and
	$$u_n\weakto\tu\neq 0\hbox{ in } \cD^{1,2}(\R^N),\;u_n(x)\to\tu(x)\hbox{ for a.e. }x\in\R^N$$ for some  $\tu\in X$. Then $u_n\to\tu$, $\tu$  is a critical point of $J_\eps$ and $J_\eps(\tu)=c_\eps$.
\end{Lem}
\begin{proof}
Note that for any $\delta>0$ there is $c_\delta>0$ such that
$$G_\eps(u+v)-G_\eps(u)-\delta |u|^{2^*}\leq c_\delta |v|^{2^*}$$
for any $u,v\in\R^N$. Hence, taking any $v\in \cC_0^{\infty}(\R^N)$ and $t\in\R$ we observe that $(G_\eps(u_n+tv)-G_\eps(u_n))$ is uniformly integrable and tight. In view of Vitali's convergence theorem and passing to a subsequence
$$\lim_{n\to\infty}\int_{\R^N}G_\eps(u_n+tv)\,dx
=\lim_{n\to\infty}\int_{\R^N}G_\eps(u_n)\,dx+\int_{\R^N}G_\eps(\tu+tv)\,dx-\int_{\R^N}G_\eps(\tu)\,dx.$$
Since $u_n\in\cM_\eps$ and
$$J_\eps(u_n)=\Big(\frac{1}{2}-\frac{1}{2^*}\Big)2^*\int_{\R^N}G_\eps(u_n)\, dx\to c_\eps$$
we get $A:=\lim_{n\to\infty}\int_{\R^N}G_\eps(u_n)\,dx=\Big(\frac{1}{2}-\frac{1}{2^*}\Big)^{-1}(2^*)^{-1}c_\eps>0$.  Note also that  $u_n+tv\in\cP$ for a.e. $n$ and for sufficiently small $|t|$.
Then
	\begin{eqnarray}\label{eq:thetain1}
		&&\lim_{t\to 0}\lim_{n\to \infty}\frac{1}{t}\Big(\Big(\int_{\R^N}G_\eps(u_n+tv)\, dx\Big)^{\frac{N-2}{N}}-
		\Big(\int_{\R^N}G_\eps(u_n)\, dx\Big)^{\frac{N-2}{N}}\Big)\\\nonumber
			&&=\lim_{t\to 0}\Big(\Big(A+\int_{\R^N}G_\eps(\tu+tv)\,dx-\int_{\R^N}G_\eps(\tu)\,dx\Big)^{\frac{N-2}{N}}-
			A^{\frac{N-2}{N}}\Big)\\\nonumber
		&&=\frac{N-2}{N}
		A^{-\frac{2}{N}}
		\Big(\int_{\R^N}g_\eps(\tu)(v)\, dx\Big)\\\nonumber
		&&=\frac{N-2}{N}
		\Big(\frac{1}{2}-\frac{1}{2^*}\Big)^{\frac{2}{N}}(2^*)^{\frac{2}{N}}c_\eps^{-\frac{2}{N}}
		\Big(\int_{\R^N}g_\eps(\tu)v\, dx\Big).
	\end{eqnarray}
Observe that if $u_n+tv\in\cP$ then $J_\eps(m_{\cP}(u_n+tv))\geq c_\eps$, that is
	$$r(u_n+tv)^{2-N}\Big(\frac{1}{2}-\frac{1}{2^*}\Big)\int_{\R^N}|\nabla (u_n+tv)|^2\, dx\geq c_\eps.$$
	Hence
	\begin{eqnarray*}
		\Big(\frac{1}{2}-\frac{1}{2^*}\Big)^{\frac2N}
		\int_{\R^N}|\nabla (u_n+tv)|^2\, dx&\geq&	c_\eps^{\frac2N}\Big(2^*\int_{\R^N}G_\eps(u_n+tv)\, dx\Big)^{\frac{N-2}{N}}
	\end{eqnarray*}
	and by \eqref{eq:thetain1} we obtain
	\begin{eqnarray*}
		&&\lim_{n\to\infty}\int_{\R^N}\langle \nabla u_n,\nabla v \rangle\, dx
		=\lim_{t\to 0}\lim_{n\to\infty}\frac{1}{2t}\Big(\int_{\R^N}|\nabla (u_n+tv)|^2\, dx-\int_{\R^N}|\nabla u_n|^2\, dx\Big)\\
		&&\geq \lim_{t\to 0}\lim_{n\to\infty}\Big(\frac{1}{2}-\frac{1}{2^*}\Big)^{-\frac{2}{N}}c_\eps^{\frac{2}{N}}\frac{1}{2t}\Big(
		 \Big(2^*\int_{\R^N}G(u_n+tv)\, dx\Big)^{\frac{N-2}{N}}
		-\Big(2^*\int_{\R^N}G(u_n)\, dx\Big)^{\frac{N-2}{N}}\Big)\\
		&&=\int_{\R^N}g_\eps(\tu)v\, dx.
	\end{eqnarray*}
	Thus 
	$$\int_{\R^N}\langle \nabla \tu,\nabla v \rangle\, dx \geq \int_{\R^N}g_\eps(\tu)v\, dx$$
	for any $v\in \cC_0^{\infty}(\R^N)$ and we infer that $\tu$ is a critical point of $J_\eps$. 
In view of the Pohozaev identity (cf. Proposition \ref{PropPohozaev}), $\tu\in\cM_\eps$, $m_{\cP}(\tu)=\tu$ and
$$c_\eps\leq J(\tu)=\Big(\frac{1}{2}-\frac{1}{2^*}\Big)\int_{\R^N}|\nabla \tu|^2\,dx\leq 
\liminf_{n\to\infty}\Big(\frac{1}{2}-\frac{1}{2^*}\Big)\int_{\R^N}|\nabla u_n|^2\,dx=c_\eps.$$
Therefore $\|u_n\|\to \|\tu\|$ and $u_n\to\tu$.
\end{proof}

\begin{altproof}{Theorem \ref{ThMain1}}
(a)	Let $(u_n)\subset \cM_\eps$ be a minimizing sequence of $J_\eps$. i.e. $J_\eps(u_n)\to c_\eps$. 
	Since $J_\eps$ is coercive on $\cM_\eps$, $(u_n)$ is bounded. Observe that
	\begin{equation}\label{eqineq123}
	2^*\int_{\R^N}G_{+}(u_n)\, dx\geq \int_{\R^N}|\nabla u_n|^2\, dx= \Big(\frac12-\frac{1}{2^*}\Big)^{-1}c_\eps+o(1),
	\end{equation}
	\begin{equation*}
	\lim_{s\to 0}G_{+}(s)/s^{2^*}=\lim_{|s|\to \infty}G_{+}(s)/s^{2^*}=0,
	\end{equation*}
	and in view of Lemma \ref{lem:Lions}, \eqref{eq:LionsCond11} is not satisfied. Therefore, passing to a subsequence, we find $u_\eps\in \cD^{1,2}(\R^N)$ and $(y_n)\subset \R^N$ such that  $$u_n(\cdot+y_n)\weakto u_\eps\neq 0\;\hbox{ and }u_n(x+y_n)\to u_\eps(x)$$ for a.e. $x\in\R^N$ as $n\to \infty$. By Lemma \ref{lem:theta} we infer that $u_\eps\in\cM_\eps$ is a critical point of $J_\eps$ at level $c_\eps$.
	Now we let $\eps\to 0$ and in order to avoid confusion with notation, we denote the dependence of $\cP$ and $m_\cP$ on $\eps$  by $\cP_\eps$ and $m_{\cP_\eps}$ respectively.
	Take any $u\in\cM$ and
	observe that
	\begin{eqnarray}
		J_\eps(u_\eps)&\leq&\nonumber
		J_\eps(m_{\cP_\eps}(u))=\Big(\frac12-\frac{1}{2^*}\Big)\Big(\int_{\R^N}|\nabla u|^2\, dx\Big)^{\frac{N}{2}}\Big(2^*\int_{\R^N}G_{+}(u)-G_-^\eps(u)\, dx\Big)^{-\frac{N-2}{2}}\\
		&\leq&\Big(\frac12-\frac{1}{2^*}\Big)\Big(\int_{\R^N}|\nabla u|^2\,\nonumber dx\Big)^{\frac{N}{2}}\Big(2^*\int_{\R^N}G(u)\, dx\Big)^{-\frac{N-2}{2}}\\
		&=& J(u)\label{ineqJepsJ}
	\end{eqnarray}
	Hence $$J_\eps(u_\eps)\leq \inf_{\cM}J$$ 
	and 
	\begin{equation}\label{eq:ineqinfm}
	\int_{\R^N}|\nabla u_\eps|^2\, dx\leq \Big(\frac12-\frac{1}{2^*}\Big)^{-1}\inf_{\cM}J.
	\end{equation}
	Moreover, 
	$$J_\eps(u_\eps)=J_\eps(m_{\cP_\eps}(u_\eps))\geq J_{1/2}(m_{\cP_{1/2}}(u_\eps))\geq J_{1/2}(u_{1/2})$$
	and we obtain
	$$2^*\int_{\R^N}G_{+}(u_\eps)\, dx\geq \int_{\R^N}|\nabla u_\eps|^2\, dx\geq \Big(\frac12-\frac{1}{2^*}\Big)^{-1}J_{1/2}(u_{1/2})$$
	for $\eps\in (0,1/2]$.
	Since $u_\eps$ is bounded in $\cD^{1,2}(\R^N)$ and $\int_{\R^N}G_{+}(u_\eps)\, dx$ is bounded away from $0$, 
	in view of Lemma \ref{lem:Lions} we infer that \eqref{eq:LionsCond11} does not hold. Therefore,
	passing to a subsequence and up to a translation, we may assume that $u_\eps\weakto u_0\neq 0$ and $u_\eps(x)\to u_0(x)$ for a.e. $x\in\R^N$ as $\eps\to 0$. Observe that for any $v\in\cC_0^{\infty}(\R^N)$ we easily see that 
	$$\Big|\frac{1}{\eps^{2^*-1}}|u_\eps|^{2^*-1}\chi_{\{|u_\eps|\leq \eps\}}g_-(u_\eps)v\Big|\leq 
	\Big|\chi_{\{|u_\eps|\leq \eps\}}g_-(u_\eps)v\Big|\to 0\hbox{ a.e. on }\R^N$$
	and the family $\big\{g_-(u_\eps)v\big\}$ is uniformly integrable, since $|g_-(u_\eps)v|\leq c(1+|u_\eps|^{2^*-1})|v|$ for a constant $c>0$. Since
	$$2^*\int_{\R^N}G^\eps_{-}(u_\eps)\,dx=2^*\int_{\R^N}G_{+}(u_\eps)\,dx-\int_{\R^N}|\nabla u_\eps|^2\,dx$$
	is bounded, by Fatou's lemma we infer that $G_{-}(u_0)\in L^1(\R^N)$.
	Then, in view of the Vitali convergence theorem
	\begin{eqnarray*}
		J_\eps'(u_\eps)(v)&=&\int_{\R^N}\langle \nabla u_\eps,\nabla v\rangle\,dx-\int_{\R^N}g_{+}(u_\eps)v\,dx+\int_{\R^N}\vp_{\eps}(u_\eps)g_{-}(u_\eps)v\,dx\\
		&\to& J'(u_0)(v)
	\end{eqnarray*}
	and $u_0$ is a nontrivial weak solution of \eqref{eq}, 
	 and
	by the Pohozaev identity in Proposition \ref{PropPohozaev}, $u_0\in\cM$.  Taking into account
	\eqref{eq:ineqinfm},
	\begin{eqnarray*}
		J(u_0)&=& \Big(\frac12-\frac{1}{2^*}\Big)\int_{\R^N}|\nabla u_0|^2\leq \Big(\frac12-\frac{1}{2^*}\Big)\liminf_{\eps\to 0}\int_{\R^N}|\nabla u_\eps|^2\,dx\\
		&\leq&  \inf_{\cM} J,
	\end{eqnarray*}
	hence $J(u_0)=\inf_{\cM} J$. Now suppose that $g$ is odd. Then $G_{+}$ and $G_{-}$ are even. Observe that for the minimizing sequence  $(u_n)\subset \cM_\eps$ we can consider $(|u_n|)\subset \cM_\eps$ and 
	$$J_\eps(|u_n|)=\Big(\frac12-\frac{1}{2^*}\Big)\int_{\R^N}|\nabla |u_n||^2\,dx=J_\eps(u_n).$$
	Hence $(|u_n|)$ is a minimizing sequence of $J_\eps$ and therefore we can assume that $u_\eps\geq 0$. Hence $u_0\geq 0$ and
	in view of  the strong maximum principle $u_0>0$.\\
(b) Suppose that $J(u)=c:=\inf_{\cM}J$. Note that $G(u+v)\in L^1(\R^N)$ for any $v\in\cC_0^{\infty}(\R^N)$. Let us fix $v\in\cC_0^{\infty}(\R^N)$ and similarly as in proof of Lemma \ref{lem:theta} we show that by the Vitaly convergence theorem
\begin{eqnarray*}
&&\lim_{t\to 0}\frac1t\Big(\Big(2^*\int_{\R^N}G(u+tv)\, dx\Big)^{\frac{N-2}{N}}-
\Big(2^*\int_{\R^N}G(u)\, dx\Big)^{\frac{N-2}{N}}\Big)\\\nonumber
&&=\frac{N-2}{N}
\Big(\frac{1}{2}-\frac{1}{2^*}\Big)^{\frac{2}{N}}c^{-\frac{2}{N}}
\Big(2^*\int_{\R^N}g(u)v\, dx\Big).
\end{eqnarray*} 
Note that 
$$\int_{\R^N}G(u+tv)\, dx >0$$
if $|t|$ is sufficiently small. Hence $(u+tv)(r\cdot)\in\cM$ for 
$r=\Big(2^*\int_{\R^N}G(u+tv)\, dx\Big)^{1/2}/\|u\|$,
$J\big((u+tv)(r\cdot)\big)\geq c$, i.e.
\begin{eqnarray*}
	\Big(\frac{1}{2}-\frac{1}{2^*}\Big)^{\frac2N}
	\int_{\R^N}|\nabla (u+tv)|^2\, dx&\geq&	c^{\frac2N}\Big(2^*\int_{\R^N}G(u+tv)\, dx\Big)^{\frac{N-2}{N}}.
\end{eqnarray*}
Similarly as in proof of Lemma \ref{lem:theta} we show that $J'(u)(v)=0$. Therefore $u$ is a weak solution of \eqref{eq}. Take $\lambda:=\int_{\R^N} G(u)\, dx=\frac{1}{2^*}\|u\|^2>0$. Then, for any $v\in\cD^{1,2}(\R^N)$ such that  
\begin{equation}\label{eq:constaint}
\int_{\R^N} G(v)\, dx=\lambda
\end{equation}
we get $v(r\cdot)\in\cM$ for $r:=(2^*\lambda)^{1/2}/\|v\|$. Hence
$J(v(r\cdot))\geq J(u)$, 
$$\Big(\frac12-\frac{1}{2^*}\Big)r^{2-N}\|v\|^2\geq \Big(\frac12-\frac{1}{2^*}\Big)\|u\|^2,$$
and we get
$$\|v\|^2\geq \|u\|^2.$$
Therefore $u$ is a minimizer of the functional $\cD^{1,2}(\R^N)\ni v\mapsto \|v\|^2\in\R$ under the constraint \eqref{eq:constaint}. In view of Mariş \cite{Maris}[Theorem 2], $u$ is radial up to a translation.\\
(c) Take any $u\in\cD^{1,2}(\R^N)$ such that $\int_{\R^N}G(u)\,dx>0$. Then $u(r\cdot )\in \cM$ for some $r>0$ and the inequality $J(u(r\cdot))\geq \inf_{\cN}J$ is equivalent to \eqref{eq:ineq} with 
$C_{N,G}=2^*\Big(\frac12-\frac{1}{2^*}\Big)^{-\frac{2}{N-2}}(\inf_{\cM} J)^{\frac{2}{N-2}}$. Clearly, if $u\in\cM$ and $J(u)=\inf_{\cM} J$, then
$u$ is a minimizer of \eqref{eq:ineq}.\\ 
\indent Now let $u$ be a minimizer of \eqref{eq:ineq}. 
Then $\int_{\R^N}G(u)\,dx>0$ and $u (\lambda\cdot)\in\cM$ for a unique $\lambda>0$ and $J(u(\lambda\cdot))=\inf_{\cM} J$.
\end{altproof}

\begin{altproof}{Corollary \ref{th:cubicquintic}}
(a) follows from Theorem \ref{ThMain1} (a).\\
(b) Observe that $G(s)$ has nonpositive values for $m\geq m_0$ and in view of \eqref{eq:Poho}, \eqref{eq} does not have any nontrivial solutions. Similarly combining \eqref{eq:Poho} with $J'(u)(u)=0$ we infer that there are nontrivial solutions also for $m\leq 0$. For the uniqueness see \cite{Lewin} and references therein.
\end{altproof}

\section{Proofs of Theorem  \ref{ThMain2} and  Theorem \ref{ThMain3}}\label{sec:proof2}

Now, let us consider $\cO_1$-invariant functions.

\begin{altproof}{Theorem \ref{ThMain2}}
Assume that $X:=\cD^{1,2}_{\cO_1}(\R^N)\cap X_\tau$ and $2\leq m< N/2$.
Let $(u_n)\subset \cM_\eps\cap X$ be a sequence such that 
$J_\eps(u_n)\to \beta$ with 
$$\beta:=\inf_{\cM_\eps\cap X}J_\eps.$$
Since $J_\eps$ is coercive on $\cM_\eps$, $(u_n)$ is bounded. Observe that
$$2^*\int_{\R^N}G_{+}(u_n)\, dx\geq \int_{\R^N}|\nabla u_n|^2\, dx= \Big(\frac12-\frac{1}{2^*}\Big)^{-1}\beta+o(1)$$
and in view of Corollary \ref{CorLions1}, passing to a subsequence, we find $(y_n)\subset \{0\}\times\{0\}\times \R^{N-2m}$ such that  $$u_n(\cdot+y_n)\weakto u_\eps\neq 0\;\hbox{ and }u_n(x+y_n)\to u_\eps(x)$$ for a.e. $x\in\R^N$ as $n\to \infty$. 
Similarly as in proof of Lemma \ref{lem:theta} we show that $u_\eps$ is a critical point of $J_\eps|_X$ and by the Palais principle of symmetric criticality \cite{Palais}, $J_\eps'(u_\eps)=0$. By the Pohozaev identity (cf. Proposition \ref{PropPohozaev}), $u_\eps\in\cM_\eps\cap X$, $m_{\cP}(u_\eps)=u_\eps$ and
$$\beta\leq J_\eps(u_\eps)=\Big(\frac{1}{2}-\frac{1}{2^*}\Big)\int_{\R^N}|\nabla u_\eps|^2\,dx\leq 
\liminf_{n\to\infty}\Big(\frac{1}{2}-\frac{1}{2^*}\Big)\int_{\R^N}|\nabla u_n(\cdot+y_n)|^2\,dx=\beta.$$
Letting $\eps\to 0$ as in proof of Theorem \ref{ThMain1}, we find 
	 a critical point $u\in \cM\cap X$ of $J$ such that 
	$$J(u)=\inf_{\cM\cap X}J.$$
	In view of the Palais principle of symmetric criticality \cite{Palais}, $u$ solves \eqref{eq}.
	Let
	\begin{eqnarray*}
		\Omega_1&:=&\{x\in\R^N: |x_1|>|x_2|\},\\
		\Omega_2&:=&\{x\in\R^N: |x_1|<|x_2|\}.\\
	\end{eqnarray*}
	Since $u\in X_\tau\cap \cD^{1,2}_{\cO_1}(\R^N)$, we get $\chi_{\Omega_1} u\in \cD^{1,2}(\R^N)$ and 
	$\chi_{\Omega_2} u\in \cD^{1,2}(\R^N)$. Moreover $\chi_{\Omega_1} u\in \cM$ and
	$$J(u)=J(\chi_{\Omega_1} u)+J(\chi_{\Omega_2} u)=2J(\chi_{\Omega_1} u)\geq 2 \inf_{\cM} J.$$
	Suppose that $J(u)=2 \inf_{\cM} J$. Then 
	$$J(\chi_{\Omega_1} u)=\inf_{\cM} J$$
	and in view of Theorem \ref{ThMain1} (b), $\chi_{\Omega_1} u$ is radial (up to a translation), which is a contradiction.
This completes proof of \eqref{eq:thmain2}. The remaining case $2\leq m=N/2$ is contained in Theorem \ref{ThMain3}.
\end{altproof}

Now let us consider $\cO_2$-invariant functions. In order to the get the multiplicity of critical points, we need to modify $J_\eps$ in order to ensure that \eqref{eq:gelam} and \eqref{eq:gepslambda} below are satisfied. Take any even function $\psi_\lambda:\R\to [0,1]$ of class $\cC^1$ such that $\psi_\lambda(s)=1$ for $\lambda\leq |s|\leq 1/\lambda$ and $\supp(\psi_\lambda)$ is compact and does not contain $0$ for $\lambda \in (0,1]$. We set $\psi_0\equiv 1$. Let $G_{+,\lambda}(s):=\psi_\lambda(s)G_{+}(s)$ and
instead of $G_\eps$ we consider now
$$G_{(\eps,\lambda)}(s):=G_{+,\lambda}(s)-\lambda|s|^{2^*}-G^\eps_{-}(s).$$
Take $g_{+,\lambda}(s):=(\psi_\lambda(s)G_{+}(s))'$ and we check that
\begin{equation}\label{eq:gelam}
\lim_{s\to 0} g_{+,\lambda}(s)/|s|^{2^*-1}=\lim_{|s|\to \infty}g_{+,\lambda}(s)/|s|^{2^*-1}=0.
\end{equation}
Let us introduce the following functional $$J_{(\eps,\lambda)}(u):=\frac12\|u\|^2-\int_{\R^N}G_{(\eps,\lambda)}(u)\,dx$$
for $\eps\in (0,1/2]$ and $\lambda\in [0,1]$. Since  \eqref{eq:gelam} holds, $J_{(\eps,\lambda)}$ is of class $\cC^1$. Clearly, Proposition \ref{prop:defMPSU} holds if we replace $J_\eps$, $g_\eps$ and $G_\eps$ by $J_{(\eps,\lambda)}$, $g_{(\eps,\lambda)}:=G_{(\eps,\lambda)}'$ and $G_{(\eps,\lambda)}$ respectively and 
$\lambda>0$ is sufficiently small, i.e. there is $\lambda_0\in (0,1]$ such that  $G_{(0,\lambda)}(\xi_0)>0$ for $\lambda\in [0,\lambda_0]$. We may also assume that $\psi_{\lambda}(s)\geq \psi_{\lambda_0}(s)$, hence  $G_{(0,\lambda)}(s)\geq G_{(0,\lambda_0)}(s)$ for any $s\in\R$ and $\lambda\in [0,\lambda_0]$.
Here and what follows $\cP$, $\cU$, $m$ depend on $\eps$ and $\lambda$, and are given in Proposition \ref{prop:defMPSU}, where $J_\eps$, $g_\eps$ and $G_\eps$ are replaced by $J_{(\eps,\lambda)}$, $g_{(\eps,\lambda)}$ and $G_{(\eps,\lambda)}$ respectively. $\cM_{(\eps,\lambda)}$ stands for the Pohozaev manifold for $J_{(\eps,\lambda)}$.

\begin{Lem}\label{lem:Mcond3}
	Suppose that $X:=\cD^{1,2}_{\cO_2}(\R^N)\cap X_\tau$ and $(u_n)\subset\cU\cap X$ is a $(PS)_\beta$-sequence of $(J_{(\eps,\lambda)}|_X\circ m|_{\U\cap X})$ at level $\beta\in\R$, i.e. $$(J_{(\eps,\lambda)}|_X\circ m|_{\U\cap X})'(u_n)\to 0\hbox{ and }(J_{(\eps,\lambda)}|_X\circ m|_{\U\cap X})(u_n)\to\beta.$$
	(i) Then, passing to a subsequence, $u_n\to u_0$ for some $u_0\in \cU\cap X$.\\
	(ii) $J_{(\eps,\lambda)}'(m(u_0))=0$ provided that $\lambda\in (0,\lambda_0]$. 
\end{Lem}
\begin{proof}
Note that, if $\beta=\inf_{\cM_\eps\cap X} J_{(\eps,\lambda)}$, then we can argue as in Lemma \ref{lem:theta}. 
	Let $(u_n)\subset \cU\cap X$ be a sequence such that $(J_{(\eps,\lambda)}|_X\circ m|_{\U\cap X})'(u_n)\to 0$ and $(J_{(\eps,\lambda)}|_X\circ m|_{\U\cap X})(u_n)\to\beta$.  Observe that $\beta\geq \inf_{\cM_\eps\cap X} J_{(\eps,\lambda)}>0$.
	Since $J_{(\eps,\lambda)}$ is coercive on $\cM_{(\eps,\lambda)}$, $(m(u_n))$ is bounded and, passing to subsequence,  we may assume that $m(u_n)\weakto\tu$ and
	$m(u_n)(x)\weakto\tu(x)$ for a.e. $x\in\R^N$.  
	In view of  Lemma \ref{lem:BrezisLieb} (b)
we infer that
	\begin{equation}\label{eq:G_1conv}
	\int_{\R^N}G_{+,\lambda}(m(u_n))\, dx\to \int_{\R^N}G_{+,\lambda}(\tu)\, dx
	\end{equation}
	as $n\to\infty$. If $\tu=0$, then we get a contradiction with the following inequality
	$$2^*\int_{\R^N}G_{+,\lambda}(m(u_n))\, dx\geq \int_{\R^N}|\nabla m(u_n)|^2\, dx= \Big(\frac12-\frac{1}{2^*}\Big)^{-1}\beta+o(1).$$
	 Therefore $\tu\neq 0$ and we easily check  that $r(u_n)$ given by \eqref{eq:defOfR} is bounded and bounded away from $0$. 
	 For any $v\in X$ we set $v_n:=v(r(u_n)^{-1}\cdot)$ and
	 we find the following decomposition
	 $$v_n=\Big(\int_{\R^N}\langle \nabla u_n,\nabla v_n\rangle\,dx\Big)u_n+\tv_n$$
	 with 
	 $$\tv_n\in T_{u_n}\cS:=\Big\{u\in\cD^{1,2}(\R^N): \int_{\R^N}\langle\nabla u_n,\nabla u\rangle\,dx=0\Big\}.$$
	 Clearly $(\tv_n)\subset X$ is bounded and  $(J_{(\eps,\lambda)}|_X\circ m|_{\U\cap X})'(u_n)(\tv_n)\to 0$ as $n\to\infty$. Since 
	 $$\int_{\R^N}\langle \nabla u_n,\nabla v_n\rangle\,dx=r(u_n)^{N-2}
	 \int_{\R^N}\langle \nabla m(u_n),\nabla v\rangle\,dx\to 0$$
	 for any $v\in X$ such that $\int_{\R^N}\langle \nabla\tu,\nabla v\rangle\,dx=0$,
we get
	 \begin{eqnarray*}
	\big(J_{(\eps,\lambda)}|_X\circ m|_{\U\cap X}\big)'(u_n)(v(r(u_n)^{-1}\cdot))&=&\Big(\int_{\R^N}\langle \nabla u_n,\nabla v_n\rangle\,dx\Big)\big(J_{(\eps,\lambda)}|_X\circ m|_{\U\cap X}\big)'(u_n)(u_n)\\
	&&+\big(J_{(\eps,\lambda)}|_X\circ m|_{\U\cap X}\big)'(u_n)(\tv_n)\\
	&\to& 0.
	 \end{eqnarray*}
	 By Proposition \eqref{prop:defMPSU} (ii) we obtain
	 \begin{equation}\label{eq:Jepslem}
	 J_{(\eps,\lambda)}'(\tu)(v)=\lim_{n\to\infty}J_{(\eps,\lambda)}'(m(u_n))(v)=\lim_{n\to\infty}\big(J_{(\eps,\lambda)}|_X\circ m|_{\U\cap X}\big)'(u_n)(v(r(u_n)^{-1}\cdot))=0
	 \end{equation}
	for $v\in X$ such that $\int_{\R^N}\langle \nabla\tu,\nabla v\rangle\,dx=0$.
Now	we define a linear map
$\xi:X\to \R$ by the following formula
\begin{eqnarray*}
	\xi(v)&=&\int_{\R^N}\langle \nabla \tu,\nabla v\rangle\, dx -\int_{\R^N}g_{(\eps,\lambda)}(\tu)v\,dx\\
	&&-\Big(\int_{\R^N}|\nabla \tu|^2\, dx -\int_{\R^N}g_{(\eps,\lambda)}(\tu)\tu\,dx\Big)
	\|\tu\|^{-2}\int_{\R^N}\langle\nabla \tu,\nabla v\rangle\, dx
\end{eqnarray*}
and observe that $\xi(\tu)=0$. Since any $v\in X$ has the following decomposition
$$v=\Big(\int_{\R^N}\langle \nabla \tu,\nabla v\rangle\,dx\Big)\|\tu\|^{-2}\tu+\tv,\hbox{ where }\int_{\R^N}\langle \nabla\tu,\nabla\tv\rangle\,dx=0,$$
in view of \eqref{eq:Jepslem} we infer that $\xi\equiv 0$. 
Hence by the Palais principle of symmetric criticality \cite{Palais},
$\tu$ is a weak solution of the problem 
\begin{equation}\label{eq:theta_eq}
-\theta\Delta\tu=g_{(\eps,\lambda)}(\tu)
\end{equation}
with 
$$\theta=1-\Big(\int_{\R^N}|\nabla \tu|^2\, dx -\int_{\R^N}g_{(\eps,\lambda)}(\tu)\tu\,dx\Big)\|\tu\|^{-2}
=\|\tu\|^{-2}\int_{\R^N}g_{(\eps,\lambda)}(\tu)\tu\,dx.$$
Moreover, similarly as above we define linear maps
$\xi_n:X\to \R$ by the following formula
\begin{eqnarray*}
	\xi_n(v)&=&\int_{\R^N}\langle \nabla m(u_n),\nabla v\rangle\, dx -\int_{\R^N}g_{(\eps,\lambda)}(m(u_n))v\,dx\\
	&&-\Big(\int_{\R^N}|\nabla m(u_n)|^2\, dx -\int_{\R^N}g_{(\eps,\lambda)}(m(u_n))m(u_n)\,dx\Big)
	\|m(u_n)\|^{-2}\int_{\R^N}\langle\nabla m(u_n),\nabla v\rangle\, dx,
\end{eqnarray*}
and we show that $\xi_n\to 0$ in $X^*$. Hence, passing to a subsequence 
\begin{eqnarray*}
\theta_n&:=&1-\Big(\int_{\R^N}|\nabla m(u_n)|^2\, dx -\int_{\R^N}g_{(\eps,\lambda)}(m(u_n))m(u_n)\,dx\Big)\|m(u_n)\|^{-2}\\
&=&\|m(u_n)\|^{-2}\int_{\R^N}g_{(\eps,\lambda)}(m(u_n))m(u_n)\,dx
\end{eqnarray*}
converges to $\theta$. Since \eqref{eq:gelam} holds, in view of Lemma \ref{lem:BrezisLieb} and \eqref{eq:BL3} we infer that 
$$\int_{\R^N}g_{+,\lambda}(m(u_n))m(u_n)\,dx\to \int_{\R^N}g_{+,\lambda}(\tu)\tu\,dx$$
and  by the Fatou's lemma
$$\limsup_{n\to\infty}\int_{\R^N}g_{(\eps,\lambda)}(m(u_n))m(u_n)\,dx\leq \int_{\R^N}g_{(\eps,\lambda)}(\tu)\tu\,dx.$$
Since $\theta_n\to\theta$, we conclude that 
 $\|m(u_n)\|\to\|\tu\|$ and therefore $m(u_n)\to\tu$ and  $\tu\in \cM_{(\eps,\lambda)}$. By Proposition \ref{prop:defMPSU} (ii), $u_n\to u_0:=m^{-1}(\tu)$. 
We show that $\theta\neq 0$ provided that $\lambda>0$. By a contradiction, suppose that $\theta=0$, then $g_{(\eps,\lambda)}(\tu(x))=0$ for a.e. $x\in\R^N$. 
Take 
$\Sigma:=\{x\in\R^N: g_{(\eps,\lambda)}(\tu(x))=0\}$ 
and clearly $\R^N\setminus \Sigma$ has measure zero and let
$\Om:=\{x\in \Sigma: \tu(x)\neq 0\}$.
Suppose that
$\delta:=\inf_{x\in\Om}|\tu(x)|>0$. Since $\tu\in L^6(\R^N)\setminus\{0\}$, we infer that $\Om$ has finite positive measure, $\tu\in H^1(\R^N)$
and note that
$$\int_{\R^N}|\tu(x+h)-\tu(x)|^2\,dx\geq \delta \int_{\R^N}|\chi_\Om(x+h)-\chi_\Om(x)|^2\,dx\quad\hbox{ for any }h\in\R^N,$$
where $\chi_\Om$ is the characteristic function of $\Om$. In view of \cite{Ziemer}[Theorem 2.1.6] we infer that $\chi_\Om\in H^1(\R^N)$, hence  we get a contradiction.
Therefore we find a sequence $(x_n)\subset\R^N$ such that $\tu(x_n)\to 0$, $\tu(x_n)\neq 0$ and $g_{(\eps,\lambda)}(\tu(x_n))=0$. Again we get a contradiction, since
\begin{equation}\label{eq:gepslambda}
\limsup_{s\to 0^+}g_{(\eps,\lambda)}(s)/s^{2^*-1}\leq-\lambda<0.
\end{equation}
Therefore $\theta\neq 0$ and in view of the Pohozaev identity (cf. Proposition \ref{PropPohozaev}) we obtain that $\theta=1$,  since $\tu\in\cM_{(\eps,\lambda)}$. Hence (ii) holds.
\end{proof}

\begin{altproof}{Theorem \ref{ThMain3}}\\
(a)
Assume that $X:=\cD^{1,2}_{\cO_2}(\R^N)\cap X_\tau$.
Similarly as in proof of Theorem \ref{ThMain1} we find 
a critical point $u\in \cM\cap X$ of $J|_X$ such that 
$$J(u)=\inf_{\cM\cap X}J$$
and by the Palais principle of symmetric criticality \cite{Palais}, $u$ solves \eqref{eq}.\\
(b) {\em Step 1}. For any $\eps\in (0,1/2]$ and $\lambda\in (0,\lambda_0]$,  we show the existence of a sequence $(u^k_{(\eps,\lambda)})$ of critical points of $J_{(\eps,\lambda)}$ such that $J_{(\eps,\lambda)}(u^k_{(\eps,\lambda)})$ as $k\to\infty$. Let us fix $\lambda\in [0,\lambda_0]$.
In view of \cite{BerLionsII}[Theorem 10], for any $k\geq 1$ we find an odd continuous map
$$\tau:S^{k-1}\to H_0^1(B(0,R))\cap L^{\infty}(B(0,R))$$ 
such that $\tau(\sigma)$ is a radial function and $\tau(\sigma)\neq 0$ for all $\sigma\in S^{k-1}$, where $S^{k-1}$ is the unit sphere in $\R^k$. Moreover,  since $G_{(0,\lambda)}(\xi_0)>0$,  we may find some constants $c_2,c_3>0$ independent on $R$ such that
$$\int_{B(0,R)}G_{(0,\lambda)}(\tau(\sigma))\,dx\geq c_2R^N-c_3R^{N-1}$$
for any $\sigma\in S^{k-1}$. 
 As in \cite{MederskiBL}[Remark 4.2] we define a map 
$$\tilde{\tau}:S^{k-1}\to H_0^1(B(0,R))\cap L^{\infty}(B(0,R))$$ such that $\tilde{\tau}(\sigma)(x_1,x_2,x_3)=\tau(\sigma)(x_1,x_2,x_3)\vp(|x_1|-|x_2|)$ and   $\vp:\R\to [0,1]$ is an odd and smooth function such that $\vp(x)=1$ for $x\geq 1$,  $\vp(x)=-1$ for $x\leq -1$.  If $\lambda=\lambda_0$, then we denote this map by $\tilde{\tau_{\lambda_0}}$.  Observe that $\tilde{\tau}(\sigma)\in X$ and, again as in \cite{MederskiBL}[Remark 4.2], we show that
$$
\int_{B(0,R)}G_{(0,\lambda)}(\tilde{\tau}(\sigma))\,dx\geq \int_{B(0,R)}G_{(0,\lambda)}(\tau(\sigma))\,dx-c_1R^{N-1}$$ for $\sigma\in S^{k-1}$ and some constant $c_1>0$. Therefore, for sufficiently large $R=R(\lambda)$
\begin{equation}\label{eq:Rtau}
\int_{B(0,R)}G_{(\eps,\lambda)}(\tilde{\tau}(\sigma))\,dx\geq \int_{B(0,R)}G_{(0,\lambda)}(\tilde{\tau}(\sigma))\,dx>0
\end{equation}
for any $\eps\in[0,1/2]$ and $\lambda\in [0,\lambda_0]$.
Hence
$\tilde{\tau}(\sigma)\in\cP\cap X$ if $\eps>0$.
Taking $p(u):=u/\|u\|$ we obtain that
\begin{equation}\label{eq:gammaset}
\gamma\Big(p\big(\tilde\tau(S^{k-1})\big)\Big)\geq k,
\end{equation}
where $\gamma$ stands for the Krasnoselskii genus for closed and symmetric subsets of $X$. Therefore the Lusternik-Schnirelman values
\begin{equation}\label{eq:LSvalue}
\beta^k_{(\eps,\lambda)}:= \inf\big\{\beta\in\R:  \gamma\big(\Phi^{\beta}_{(\eps,\lambda)}\big)\geq k\big\}
\end{equation}
are finite, where $\Phi_{(\eps,\lambda)}:=J_{(\eps,\lambda)}\circ m|_X:\cU\cap X\to\R$ and
$\Phi^\beta_{(\eps,\lambda)}:=\big\{u\in \cU\cap X: \Phi_{(\eps,\lambda)}(u)\leq \beta\big\}$  for any $\eps\in (0,1/2]$ and $\lambda\in [0,\lambda_0]$. Recall that
$\cP$, $\cU$, $m$ depend on $\eps$ and $\lambda$. Moreover, observe that
\begin{equation*}
\Phi_{(\eps,\lambda)}(u)=J_{(\eps,\lambda)}(m(u))=\Big(\frac12-\frac{1}{2^*}\Big)\Big(2^*\int_{\R^N}\psi_\lambda(u)G_{+}(u)-G_{-}^\eps(u)-\frac{\lambda}{2^*}|u|^{2^*}\, dx\Big)^{-\frac{N-2}{2}},
\end{equation*}
and in view of \eqref{eq:Rtau} we obtain the following estimates
\begin{eqnarray}\label{eq:betaestimate}
\beta^k_{(1/2,0)}&\leq& \beta^k_{(\eps,0)}\leq \beta^k_{(\eps,\lambda)} \leq \beta^k_{(\eps,\lambda_0)}\\
&\leq& M^k:=\sup_{u\in p(\tilde\tau_{\lambda_0}(S^{k-1}))}\Big(\frac12-\frac{\lambda}{2^*}\Big)\Big(2^*\int_{B(0,R(\lambda_0))}G_{(0,\lambda_0)}(u)\, dx\Big)^{-\frac{N-2}{2}},\nonumber
\end{eqnarray}
for any $\eps\in (0,1/2]$ and $\lambda\in [0,\lambda_0]$. 
Since Lemma \ref{lem:Mcond3} holds, in view of
 \cite{MederskiBL}[Theorem 2.2 (c)] we get an infinite sequence of critical points, namely $(\beta^k_{(\eps,\lambda)})_{k\geq 1}$ are critical values
 provided that $\eps\in (0,1/2]$ and $\lambda\in (0,\lambda_0]$.
It is standard to show that the sequence is unbounded.
Indeed, 
as in \cite{MederskiBL,Rabinowitz:1986} we show that $\beta^1_{(\eps,\lambda)}<\beta^2_{(\eps,\lambda)}<...<\beta^k_{(\eps,\lambda)}<...$ is an increasing sequence of critical values, due to Lemma \ref{lem:Mcond3} and $\Phi_{(\eps,\lambda)}(u)\to\infty$ as $u\to u_0$ for some $u_0\in\partial (\cU\cap X)$. Suppose that  $\bar\beta:=\lim_{k\to\infty}\beta^k_{(\eps,\lambda)}<\infty$. 
Note that
$$\cK^{\bar{\beta}}:=\big\{u\in \cU\cap X: \Phi'_{(\eps,\lambda)}(u)=0\hbox{ and }\Phi_{(\eps,\lambda)}(u)=\bar{\beta}\big\}$$
is compact and $\gamma\big(\cl B(\cK^{\bar{\beta}},\delta)\big)=\gamma\big(\cK^{\bar{\beta}}\big)<\infty$ for some small $\delta>0$.
Similarly as in proof of \cite{MederskiBL}[Theorem 2.2] we construct a continuous and odd map 
$h:\Phi^{\bar{\beta}+\eta}_{(\eps,\lambda)}\setminus B(\cK^{\bar{\beta}},\delta)\to \Phi^{\bar{\beta}-\eta}_{(\eps,\lambda)}$ for sufficiently small $\eta>0$ 
such that 
$$\Phi^{\bar{\beta}+\eta}_{(\eps,\lambda)}\setminus \big(B(\cK^{\bar{\beta}},\delta)\cup \Phi^{\bar{\beta}-\eta}_{(\eps,\lambda)}\big)$$
does not contain any critical point.
Hence
\begin{eqnarray*}
\gamma\big(\Phi^{\bar\beta+\eta}_{(\eps,\lambda)}\big)&\leq& \gamma \big((\cl B(\cK^{\bar\beta},\delta)\big)+
\gamma\big(\Phi^{\bar\beta+\eta}_{(\eps,\lambda)}\setminus B(\cK^{\bar\beta},\delta)\big)\\
&\leq &\gamma \big(\cl B(\cK^{\bar\beta},\delta)\big)+
\gamma\big(\Phi^{\bar\beta-\eta}_{(\eps,\lambda)}\big)=:l<\infty.
\end{eqnarray*}
We obtain a contradiction with $\gamma\big(\Phi^{\bar\beta+\eta}_{(\eps,\lambda)}\big)\geq \gamma\big(\Phi^{\beta^{l+2}_{(\eps,\lambda)}}_{(\eps,\lambda)}\big)\geq l+1$. Therefore $\Phi_{(\eps,\lambda)}$ has a sequence of critical points $(u^k_{(\eps,\lambda)})\subset \cS$ with $$\Phi_{(\eps,\lambda)}\big((u^k_{(\eps,\lambda)})\big)=\beta^k_{(\eps,\lambda)}\to\infty$$
as $k\to\infty$, for $\eps\in(0,1/2]$ and $\lambda\in (0,\lambda_0]$.
Hence, by Lemma \ref{lem:Mcond3} (ii), $J_{(\eps,\lambda)}$ has an unbounded sequence of critical points $(m(u^k_{(\eps,\lambda)}))$ for $\eps\in(0,1/2]$ and $\lambda\in (0,\lambda_0]$.\\
{\em Step 2}. We show the existence a sequence $(u_\eps^k)$ of critical points of $J_\eps$  for any $\eps\in (0,1/2]$ such that $J_\eps(u_\eps^k)\to\infty$ as $k\to\infty$. Indeed, take $\lambda_n\in (0,\lambda_0]$ such that $\lambda_n\to 0$ as $n\to\infty$ and 
in view of \eqref{eq:betaestimate}, $v_n:=m(u^k_{(\eps,\lambda_n)})$ is bounded.
Passing to a subsequence, $v_n\weakto v_0$ and $v_n(x)\to v_0(x)$ for a.e. $x\in\R^N$. Since $J_{(\eps,\lambda_n)}'(v_n)=0$, we obtain  that $J'_\eps(v_0)=0$ 
and
by Lemma \ref{lem:BrezisLieb} (b)
$$
	\int_{\R^N}G_{+}(v_n)\, dx\to \int_{\R^N}G_{+}(v_0)\, dx
$$
	as $n\to\infty$. If $v_0=0$, then we get a contradiction since 
$$0<\liminf_{n\to\infty}\Phi_{(\eps,\lambda_n)}(u^k_{(\eps,\lambda_n)})\leq \liminf_{n\to\infty}\Big(\frac12-\frac{1}{2^*}\Big)\int_{\R^N}G_{+}(v_n)-G_-^\eps(v_n)\, dx\leq 0.$$
Therefore $v_0\in \cM_{\eps}$ and $u^k_\eps:=m^{-1}(v_0)$ is a critical point of $\Phi_{(\eps,0)}$. Moreover by Fatou's lemma
\begin{eqnarray*}
\|v_0\|^2+2^*\int_{\R^N}G_{-}^\eps(v_0)\,dx&\leq&2^*\liminf_{n\to\infty}\Big(\|v_n\|^2+2^*\int_{\R^N}G_{-}^\eps(v_n)\,dx\Big)\\
&\leq&2^*\liminf_{n\to\infty}\int_{\R^N}G_{+}(v_n)\,dx=2^*\int_{\R^N}G_{+}(v_0)\,dx\\
&=&\|v_0\|^2+2^*\int_{\R^N}G_{-}^\eps(v_0)\,dx,
\end{eqnarray*}
hence $v_n\to v_0$. Therefore $u^k_{(\eps,\lambda_n)}\to u^k_{\eps}$ and $\beta^k_{(\eps,\lambda_n)}\to \Phi_\eps(u^k_{\eps})$ as $n\to\infty$. Moreover $J_\eps(m(u_\eps^k))=\Phi_\eps(u^k_{\eps})\geq \beta^k_{(\eps,0)}\to\infty$ as $k\to\infty$.\\
{\em Step 3}. We show the existence of an unbounded sequence of critical point of $J$ with finite energy. 
Take $\eps_n\in (0,1/2]$ such that $\eps_n\to 0$ as $n\to\infty$. Again, in view of \eqref{eq:betaestimate} and
passing to a subsequence, we may assume that $v_n:=m(u^k_{\eps_n})\weakto v^k$ and $v_n(x)\to v^k(x)$ for a.e. $x\in\R^N$. Since $J_{\eps_n}'(v_n)=0$, we obtain  that $J'(v^k)(\psi)=0$ for any $\psi\in\cC_0^{\infty}(\R^N)$, 
and
by Lemma \ref{lem:BrezisLieb} 
$$
\int_{\R^N}G_{+}(v_n)\, dx\to \int_{\R^N}G_{+}(v^k)\, dx
$$
as $n\to\infty$. If $v^k=0$, then 
$$\beta^k_{(1/2,0)}\leq \liminf_{n\to\infty}\Phi_{(\eps_n,0)}(u^k_{\eps_n})=\liminf_{n\to\infty}\Big(\frac12-\frac{1}{2^*}\Big)\int_{\R^N}G_{+}(v_n)-G_{-}^{\eps_n}(v_n)\, dx\leq0$$
and we get a contradiction since $\beta^k_{(1/2,0)}$ is a critical value and by \eqref{eq:infJM}, $$\beta^k_{(1/2,0)}\geq \inf_{\cM_{1/2}}J_{1/2}>0.$$
By the Fatou's lemma 
$$\|v^k\|^2+2^*\int_{\R^N}G_{-}(v^k)\,dx\leq\liminf_{n\to\infty}\Big(\|v_n\|^2+2^*\int_{\R^N}G_{-}^{\eps_n}(v_n)\,dx\Big)
=G_{+}(v^k)\,dx$$
and $G_{-}(v^k)\in L^1(\R^N)$. In view of Proposition \ref{PropPohozaev}, we obtain that $v^k\in\cM$, i.e. the equality holds above, hence $\|v_n\|\to \|v^k\|$. Therefore $v_n\to v^k$ and 
$$J(v^k)\geq \beta^k_{(1/2,0)}\to\infty$$ as $k\to\infty$.
\end{altproof}

\appendix
\section{Convergence results and profile decompositions}\label{sec:profile}

In our variational approach, the following lemma replaces compactness results of Strauss for radial functions \cite{BerLions}[Lemma A.I, Lemma A.III] and allows to consider a wider class of symmetric functions. Recall that $\cO\subset \cO(N)$ is a subgroup such that $\R^N$ is {\em compatible with $\cO$} (in the sense of  \cite{Willem}[Definition 1.23], cf. \cite{Lions82}), if
for some $r>0$ 
$$\lim_{|y|\to\infty}m(y,r)=\infty,$$
where
$$m(y,r):=\sup\big\{n\in\N: \hbox{there exist }g_1,...,g_n\in\cO\hbox{ such that } B(g_iy,r)\cap B(g_jy,r)=\emptyset\hbox{ for }i\neq j\big\}$$
and $y\in\R^N$.
For instance $\R^N$ is compatible with $\cO(N)$ and with $\cO_2$.

\begin{Lem}\label{lem:BrezisLieb}
	Suppose that $(u_n)\subset \cD^{1,2}(\R^N)$ is bounded and $u_n(x)\to u_0(x)$ for a.e. $x\in\R^N$.\\ 
(a)	Then 
	\begin{equation}\label{eq:BL1}
		\lim_{n\to\infty}\Big(\int_{\R^N}\Psi(u_n)\, dx-\int_{\R^N}\Psi(u_n-u_0)\, dx\Big)= 
		\int_{\R^N}\Psi(u_0)\, dx
	\end{equation}
	for any function $\Psi:\R\to\R$ of class $\cC^1$ such that $|\Psi'(u_n)|\leq C|s|^{2^*-1}$ for any $s\in\R$ and some constant $C>0$.\\
(b)	Suppose that  $\R^N$ is compatible with $\cO\subset \cO(N)$ 
and assume that each $u_n$ is $\cO$-invariant. If, in addition, $s\mapsto |\Psi(s)|$ satisfies \eqref{eq:Psi}, then
\begin{equation}\label{eq:BL2}
\lim_{n\to\infty}\int_{\R^N}\Psi(u_n)\, dx= 
\int_{\R^N}\Psi(u_0)\, dx,
\end{equation}
and if $s\mapsto |\Psi'(s)s|$ satisfies \eqref{eq:Psi}, then
\begin{equation}\label{eq:BL3}
\lim_{n\to\infty}\int_{\R^N}\Psi'(u_n)u_n\, dx= 
	\int_{\R^N}\Psi'(u_0)u_0\, dx.
\end{equation}
\end{Lem}
\begin{proof} (a)
Observe that by Vitali's convergence theorem 
\begin{eqnarray*}
	\int_{\mathbb{R}^N}\Psi(u_n)-\Psi(u_n-u_0)\, dx
	&=&\int_{\mathbb{R}^N}\int_0^1 -\frac{d}{ds}\Psi(u_n-s u_0)\, ds\,dx
	=\int_{\mathbb{R}^N}\int_0^1 \Psi'(u_n-s u_0)u_0\,ds\, dx\\\nonumber
	&\rightarrow& \int_0^1 \int_{\mathbb{R}^N} \Psi'(\tu_0-s u_0)u_0\,dx\,ds
	=\int_{\mathbb{R}^N}\int_0^1 -\frac{d}{ds}\Psi(\tu_0-s u_0)\, ds\, dx\\
	&=&\int_{\mathbb{R}^N}\Psi(u_0)\, dx\nonumber
\end{eqnarray*}
as $n\to\infty$.\\
(b) Suppose that $\R^N$ is compatible with $\cO$ and then
$$m(y,r)\int_{B(y,r)}|u_n-u_0|^{2^*}\,dx\leq \int_{\R^N}|u_n-u_0|^{2^*}\,dx$$
is bounded. Observe that
$$\int_{B(y,r)}|u_n-u_0|^{2}\,dx\leq C\Big(\int_{B(0,r)}|(u_n-u_0)(\cdot+y)|^{2^*}\,dx\Big)^{2/2^*}\leq C|u_n-u_0|^2_{2^*} m(y,r)^{-2/2^*}$$
for some constant $C>0$. Take any $\eps>0$ and note that we find $R>0$ such that 
$$C|u_n-u_0|^2_{2^*} m(y,r)^{-2/2^*}<\eps$$ for $|y|\geq R$ and 
$$\int_{B(y,r)}|u_n-u_0|^{2}\,dx\leq \int_{B(0,r+R)}|u_n-u_0|^{2}\,dx<\eps$$
for $|y|<R$ and sufficiently large $n$. Therefore \eqref{eq:LionsCond11} holds for $u_n-u_0$ and in view of Lemma \ref{lem:Lions} we get 
$$\lim_{n\to\infty}\int_{\R^N}\Psi(u_n-u_0)\, dx=0$$
and \eqref{eq:BL2} holds.
Now observe that for any $\eps>0$, $2<p<2^*<q$ we find $0<\delta<M$ and $c_\eps>0$  such that
$$|\Psi'(s)|\leq \eps |s|^{2^*-1}\quad\hbox{ for }|s|<\delta\hbox{ and }|s|>M,$$
and 
$$|\Psi'(s)|\leq c_\eps\min\big\{|s|^{2^*(1-\frac{1}{p})},|s|^{2^*(1-\frac{1}{q})}\big\}\quad
\hbox{ for }\delta\leq |s|\leq M.$$
Then, by the Vitali convergence theorem and by \eqref{eq:BL2} applied to $\tilde\Psi(s)=\min\{|s|^p,|s|^q\}$ and $(u_n-u_0)$ we obtain
\begin{eqnarray*}
	&&\Big|\int_{\R^N}\Psi'(u_n)u_n-\Psi(u_0)u_0\,dx\Big|\leq \int_{\R^N}|\Psi'(u_n)-\Psi'(u_0)||u_0|\,dx\\
	&&\hspace{5mm}+\int_{\R^N}|\Psi'(u_n)||u_n-u_0|\,dx
	= o(1) + \int_{\R^N}|\Psi'(u_n)||u_n-u_0|\,dx\\
	&& \leq o(1) + \eps|u_n|_{2^*}^{2^*-1}|u_n-u_0|_{2^*}+ \int_{|u_n-u_0|>1}|u_n|^{2^*(1-\frac{1}{p})}|u_n-u_0|\,dx\\
\end{eqnarray*}
\begin{eqnarray*}
	&&\hspace{5mm}+ \int_{|u_n-u_0|\leq 1}|u_n|^{2^*(1-\frac{1}{q})}|u_n-u_0|\,dx\\
	&&\leq o(1) + \eps|u_n|_{2^*}^{2^*-1}|u_n-u_0|_{2^*}
	+ |u_n|_{2^*}^{2^*(1-\frac{1}{p})}\Big(\int_{\R^N}\tilde\Psi(u_n-u_0)\,dx\Big)^{\frac{1}{p}}\\
	&&\hspace{5mm}+|u_n|_{2^*}^{2^*(1-\frac{1}{q})}\Big(\int_{\R^N}\tilde\Psi(u_n-u_0)\,dx\Big)^{\frac{1}{q}}\\
	&&\leq o(1)+\eps|u_n|_{2^*}^{2^*-1}|u_n-u_0|_{2^*}.
\end{eqnarray*}
Since $\eps>0$ is arbitrary we infer that 
$$\int_{\R^N}\Psi'(u_n)u_n\,dx\to \int_{\R^N}\Psi'(u_0)u_0\,dx.$$
\end{proof}

\begin{Prop}\label{PropLionsCompatib}
	Let $\cO=\cO'\times\id\subset \cO(N)$ such that $\cO'\subset \cO(M)$ and $\R^M$ is compatible with $\cO'$ for some $0\leq M\leq N$.
	Suppose that $(u_n)\subset \cD^{1,2}_{\cO}(\R^N)$ is bounded, $r_0>0$ is such that for all $r\geq r_0$
	\begin{equation}\label{eq:LionsCond12}
	\lim_{n\to\infty}\sup_{z\in \R^{N-M}} \int_{B((0,z),r)} |u_n|^2\,dx=0.
	\end{equation}
	Then 
	$$\int_{\R^N} \Psi(u_n)\, dx\to 0\quad\hbox{as } n\to\infty$$
	for any continuous function $\Psi:\R\to [0,\infty)$ such that \eqref{eq:Psi} holds.
\end{Prop}
\begin{proof}
	Suppose that
	\begin{equation}\label{eq:LionsCond12proof1}
	\int_{B((y_n,z_n),r_1)} |u_n|^{2}\,dx\geq c>0
	\end{equation}
	for some sequence $(y_n,z_n)\subset \R^M\times\R^{N-M}$ and a constant $c$, where $r_1$ is such that $$\lim_{|y|\to\infty,\;y\in\R^M}m(y,r_1)=\infty.$$
	Then $\int_{B((y_n,z_n),r_1)} |u_n|^{2^*}\,dx$ is bounded away from $0$. Since $(u_n)$ is bounded in $L^{2^*}(\R^N)$ and
	in the family $\{B(gy_n,r_1)\}_{g\in\cO'}$ we find an increasing number of disjoint balls as $|y_n|\to\infty$, we infer that $|y_n|$ must be bounded. Then for sufficiently large $r$ one obtains
	$$\int_{B((0,z_n),r)} |u_n|^2\,dx\geq \int_{B((y_n,z_n),r_1)} |u_n|^2\,dx\geq c>0$$
	and we get a contradiction with \eqref{eq:LionsCond12}. Therefore  \eqref{eq:LionsCond11} is satisfied with $r=r_1$ and by Lemma \ref{lem:Lions} we conclude.
\end{proof}

At the end of this section we would like to mention that the above variant of Brezis-Lieb lemma \eqref{eq:BL1} and Lemma \ref{lem:Lions} allow to obtain the following profile decomposition theorem in $\cD^{1,2}(\R^N)$ in the spirit of G\'erard \cite{Gerard}, cf. \cite{Nawa}.

\begin{Th}\label{ThGerard}
	Suppose that $(u_n)\subset \cD^{1,2}(\R^N)$ is bounded.
	Then there are sequences
	$(\tu_i)_{i=0}^\infty\subset \cD^{1,2}(\R^N)$, $(y_n^i)_{i=0}^\infty\subset \R^N$ for any $n\geq 1$, such that $y_n^0=0$,
	$|y_n^i-y_n^j|\rightarrow \infty$ as $n\to\infty$ for $i\neq j$, and passing to a subsequence, the following conditions hold for any $i\geq 0$:
	\begin{eqnarray}
	\nonumber
	&& u_n(\cdot+y_n^i)\weakto \tu_i\; \hbox{ in } \cD^{1,2}(\R^N)\text{ as }n\to\infty,\\
	\label{EqSplit2a}
	&& \lim_{n\to\infty}\int_{\R^N}|\nabla u_n|^2\, dx=\sum_{j=0}^i \int_{\R^N}|\nabla\tu_j|^2\, dx+\lim_{n\to\infty}\int_{\R^N}|\nabla v_n^i|^2\, dx,
	\end{eqnarray}
	where $v_n^i:=u_n-\sum_{j=0}^i\tu_j(\cdot-y_n^j)$ and
	\begin{eqnarray}
	&& \limsup_{n\to\infty}\int_{\R^N}\Psi(u_n)\, dx= \sum_{j=0}^i
	\int_{\R^N}\Psi(\tu_j)\, dx+\limsup_{n\to\infty}\int_{\R^N}\Psi(v_n^i)\, dx	\label{EqSplit3a}
	\end{eqnarray}
	for any function $\Psi:\R\to\R$ of class $\cC^1$ such that $|\Psi'(s)|\leq C|s|^{2^*-1}$ for any $s\in\R$ and some constant $C>0$.
	Moreover, if in addition $\Psi$ satisfies \eqref{eq:Psi}, then
	\begin{equation}\label{EqSplit4a}
	\lim_{i\to\infty}\Big(\limsup_{n\to\infty}\int_{\R^N}\Psi(v_n^i)\, dx\Big)=0.
	\end{equation}
\end{Th}

\begin{proof} In order to prove \eqref{EqSplit2a}--\eqref{EqSplit4a}, we follow arguments of proof of \cite{MederskiBL}[Theorem 1.4] with some modifications. Namely,
	let $(u_n)\subset \cD^{1,2}(\R^N)$ be a bounded sequence and $\Psi$ as above. Applying Lemma \ref{lem:Lions} and up to a subsequence we find 
	$K\in \N\cup \{\infty\}$ and there is a sequence
	$(\tu_i)_{i=0}^K\subset \cD^{1,2}(\R^N)$, for $0\leq i <K+1$ ($K=\infty$ then $K+1=\infty$ as well),  there are sequences $(v_n^i)\subset \cD^{1,2}(\R^N)$, $(y_n^i)\subset \R^N$ and positive numbers $(c_i)_{i=0}^{K}, (r_i)_{i=0}^{K}$ such that $y_n^0=0$, $r_0=0$ and for any  $0\leq i<K+1$ one has
	\begin{eqnarray*}	
	&&u_n(\cdot+y_n^i)\weakto\tu_i\hbox{ in }\cD^{1,2}(\R^N)\hbox{ and }\int_{B(0,n)}|u_n(\cdot+y_n^i)-\tu_i|^2\,dx\to 0\hbox{ as }n\to\infty,\\	
	&&\tu_i\neq 0\hbox{ for }1\leq i <K+1,\\
	&&|y_n^i-y_n^j|\geq n-r_i-r_j\hbox{ for } 0\leq j\neq i< K+1 \hbox{ and sufficeintly large }n,\\
	&& v_n^{-1}:=u_n\hbox{ and }v_n^i:=v_n^{i-1}-\tu_i(\cdot-y_n^i),\\
	&&\int_{B(y_n^{i},r_i)}|v_n^{i-1}|^2\, dx \geq c_{i}\geq\frac{1}{2}\sup_{y\in\R^N}\int_{B(y,r_i)}|v_n^{i-1}|^2\, dx\\
	&&\hspace{0.1cm}	r_i\geq \max\{i,r_{i-1}\},\hbox{if }i\geq 1,\hbox{ and } c_i= \frac{3}{4}\lim_{r\to\infty}\limsup_{n\to\infty}\sup_{y\in\R^N}\int_{B(y,r)}|v_n^{i-1}|^2\,dx\nonumber
	>0
	\end{eqnarray*}
	and \eqref{EqSplit2a} is satisfied.
	Next, we prove that \eqref{EqSplit3a} holds for every $i\geq 0$ by applying \eqref{eq:BL1}. 
	If there is $i\geq 0$ such that
	\begin{equation*}
	\lim_{n\to\infty}\sup_{y\in\R^N}\int_{B(y,r)}|v_n^i|^2\, dx=0
	\end{equation*}
	for every $r\geq \max\{i,r_i\}$,
	then $K=i$. If, in addition, \eqref{eq:Psi} holds, then in view of Lemma \ref{lem:Lions} we obtain that
	$$\lim_{n\to\infty}\int_{\R^N}\Psi(v_n^i)\, dx=0$$
	and we finish the proof by setting $\tu_j=0$ for $j>i$. 
	Otherwise we have $K=\infty$ and we prove \eqref{EqSplit4a} similarly as in \cite{MederskiBL}[Theorem 1.4]. 
\end{proof}

{\bf Acknowledgements.}
The author would like to thank L. Jeanjean for his remarks concerning the approximation $J_\eps$.
He is also grateful to the members of the CRC 1173 as well as the members of the Institute of Analysis at Karlsruhe Institute of Technology (KIT), where part of this work has been done, for their invitation, support and warm hospitality. The author was partially supported by the National Science Centre, Poland (Grant No. 2017/26/E/ST1/00817) and by the Deutsche Forschungs\-gemeinschaft (DFG) through CRC 1173.

\end{document}